\documentclass[10pt,a4paper,reqno]{amsart}
\usepackage{amssymb,amsmath,amsthm,amstext,amsfonts,latexsym}
\usepackage[dvips]{graphicx}
\usepackage{color}
\usepackage{epstopdf}
\usepackage{pictexwd,dcpic}

\makeatletter \@addtoreset{equation}{section} \makeatother

\theoremstyle{plain}

\theoremstyle{plain}

\newtheorem{theorem}{Theorem }[section]
\newtheorem{proposition}[theorem]{Proposition}
\newtheorem{lemma}[theorem]{Lemma}
\newtheorem{corollary}[theorem]{Corollary}
\theoremstyle{definition} \theoremstyle{remark}
\newtheorem{remark}[theorem]{Remark}
\newtheorem{definition}[theorem]{Definition}
\DeclareMathAlphabet{\mathpzc}{OT1}{pzc}{m}{it}

\newcommand{\field}[1]{\mathbb{#1}}
\newcommand{\RR}{\field{R}}
\newcommand{\CC}{\textbf{C}}
\newcommand{\ZZ}{\field{Z}}
\newcommand{\NN}{\field{N}}
\newcommand{\diam}{\operatorname{diam}}

\newcommand{\inn}{\textnormal{in}}
\newcommand{\out}{\textnormal{out}}
\newcommand{\loc}{\textnormal{loc}}
\newcommand{\Int}{\textnormal{int}}
\newcommand{\Het}{\textnormal{Het}}
\newcommand{\Hom}{\textnormal{Hom}}
\newcommand{\sign}{\textnormal{sgn}}
\newcommand{\Max}{\textnormal{Max}}

\newcommand{\tr}{\textnormal{tr}}

\numberwithin{equation}{section}

\begin{document}
\large

\title[Boundary crisis for degenerate singular cycles]{Boundary crisis for degenerate singular cycles}

\author[A. Lohse]{Alexander Lohse}
\author[A. Rodrigues]{Alexandre Rodrigues}
\address[A. Lohse and A. Rodrigues]{Departamento de Matem\'atica, Universidade do Porto, Rua do Campo Alegre, 687, 4169-007 Porto, Portugal}
\address[A. Lohse]{Fachbereich Mathematik, Universit\"at Hamburg, Bundesstra{\ss}e 55, 20146 Hamburg, Germany\bigbreak}
\vspace{1cm}
\email{alexander.lohse@math.uni-hamburg.de}
\email{alexandre.rodrigues@fc.up.pt}

\date{\today}

\begin{abstract}
The term boundary crisis refers to the destruction or creation of a chaotic attractor when parameters vary.
The locus of a boundary crisis may contain regions of positive Lebesgue measure marking the transition from regular dynamics to the chaotic regime. This article investigates the dynamics occurring near a heteroclinic cycle involving a hyperbolic equilibrium point $E$ and a hyperbolic periodic solution $P$, such that the connection from $E$ to $P$ is of codimension one and the connection from $P$ to $E$ occurs at a quadratic tangency (also of codimension one). We study these cycles as organizing centers of two-parameter bifurcation scenarios and, depending on properties of the transition maps, we find different types of shift dynamics that appear near the cycle. Breaking one or both of the connections we further explore the bifurcation diagrams previously begun by other authors. In particular, we identify the region of crisis near the cycle, by giving information on multipulse homoclinic solutions to $E$ and $P$ as well as multipulse heteroclinic tangencies from $P$ to $E$, and bifurcating periodic solutions, giving partial answers to the problems (Q1)--(Q3) of E. Knobloch (2008), \emph{Spatially localised structures in dissipative systems: open problems}, Nonlinearity, 21, 45--60. Throughout our analysis, we focus on the case where $E$ has real eigenvalues and $P$ has positive Floquet multipliers.
\end{abstract}

\maketitle

{\small\noindent\emph{MSC 2010:} 34C37, 37C29, 37G35, 37D45}
\vspace{0.5cm}

\section{Introduction}\label{intro}
\subsection{The object of study}
This paper presents new results in the bifurcation theory of vector fields by exploring a three-dimensional heteroclinic structure called \emph{singular cycle} started in Morales and Pac\'ifico \cite{MP}. In general, a singular cycle is a finite set of hyperbolic non-trivial periodic solutions and at least one hyperbolic equilibrium, which are linked in a cyclic way by trajectories in the intersections of their respective stable and unstable manifolds. In the present work, we investigate a singular cycle between a unique equilibrium $E$ and a unique periodic solution $P$. We assume that the stable manifolds of both $E$ and $P$ are two-dimensional. Therefore, generically, the heteroclinic connection from $E$ to $P$ is of codimension one, while that from $P$ to $E$ is of codimension zero, as an intersection of two-dimensional manifolds in three-dimensional space. Such a codimension one cycle is referred to as an \emph{$EP1$-cycle} by Champneys \emph{et al.}\ \cite{Champneys2009}. We focus our attention on the case where the stable manifold of $E$ intersects the unstable manifold of $P$ tangentially, turning the cycle into a codimension two phenomenon, called \emph{$EP1t$-cycle} in \cite{Champneys2009}.

\subsection{The results}
Concentrating on the case where the linearization at the equilibrium $E$ has only real eigenvalues and $P$ has positive Floquet multipliers, we set up a compilation of local and global maps as an approximation for the dynamics in a standard way. 
 Note that while our arguments are formulated for the first return map of a model system, the dynamical results can be carried over to a neighbourhood of the singular cycle, where  $C^2$-conjugacy of both return maps holds.

 Our first main result is the characterization of the dynamics in the organizing center, Theorem \ref{main-thm}: depending on the sign of the parameters 
that determine the transitions between $E$ and $P$ (\emph{i.e.}\ inward or outward fold of the unstable manifold of $P$, inclination flip or not) we prove the existence of different types of shift dynamics near the cycle. 
This complements the results of \cite{MP} by addressing the case where the cycle is not necessarily isolated. 
In Theorem \ref{thm-hyp} we also partially answer the question of hyperbolicity for the non-wandering set associated to the singular cycle, when restricted to a compact set not containing the cycle. 

A boundary crisis is a mechanism for destroying a chaotic set when the parameters vary. The locus of boundary crises can contain regions of positive Lebesgue measure that mark the beginning of the chaotic regime; in general this mechanism is related with quasi-stochastic attractors and wild phenomena. In our two-parameter setting, these global bifurcations are associated with curves of homo- and heteroclinic tangencies of invariant saddles, and the complete understanding of the phenomenon is far from being done. 

In the present paper, we give a generic description of the boundary crisis for the cases with shift dynamics in the organizing center. In Theorems \ref{thm-E-hom} and \ref{thm-P-hom} we study the presence of multipulse homoclinic solutions, \emph{i.e}.\ trajectories that are bi-asymptotic to $E$ or $P$ and pass more than once around the heteroclinic cycle (or its remnants). They accumulate on the curves of one-pulse homoclinic solutions described in \cite{Champneys2009, MP}. Theorem \ref{thm-PE-het} contains a corresponding result for multipulse heteroclinic connections from $P$ to $E$. The homo- and heteroclinic tangencies we encounter are related to the work of Hittmeyer \emph{et al.}\ \cite[Section 2.3]{Hittmeyer}, who study 
the transition between different types of wild chaos. 

Our results add information to the bifurcation diagrams in \cite{Champneys2009, MP}, and bring them another step closer to completion. Furthermore, the results in the present paper give partial answers the open problems \textbf{(Q1)}--\textbf{(Q3)} stated by Knobloch \cite{Knobloch_open}, about the structure of the \emph{homoclinic snaking} (continuation of the homoclinic cycles to $E$ near the original cycle). 

\subsection{State of art}
This type of degenerate singular cycle arises in several applications: for instance, Champneys \emph{et al.}\ \cite{Champneys2007} study excitable systems of reaction-diffusion equations that are used to model various biophysical phenomena. In this regard, $EP1t$-cycles occur as one of various possible codimension two mechanisms for the interaction of Hopf and homoclinic bifurcation curves, corresponding to the onset of different types of waves. Apart from that, $EP1t$-cycles come up in the context of semiconductor lasers with optical reinjection, as noted by Krauskopf and Oldeman in \cite{KrauskopfOldeman} and studied further in \cite{Champneys2009, KRies}. We take a general bifurcation theoretical standpoint and are interested in how these cycles act as \emph{organizing centers} for homoclinic bifurcations associated to the equilibrium $E$ and the periodic solution $P$.

The unfolding of $EP1$-cycles has previously been explored by many authors, such as ~Bam\'on \emph{et al.}\ \cite{BLMP}, Labarca and San Mart\'in \cite{LS}, Pac\'ifico and Rovella \cite{PR} or San Mart\'in \cite{S}, for instance. Studies of the codimension two $EP1t$-cycles include \cite{Champneys2009, MP}. The second reference focuses on isolated cycles and is primarily concerned with the prevalence of hyperbolicity near an $EP1t$-cycle: they show that for generic families of vector fields passing through a vector field with an $EP1t$-cycle, the set of parameters corresponding to hyperbolic flows has full Lebesgue measure (assuming the cycle is contracting). In doing so, they extend a corresponding result for $EP1$-cycles in \cite{LS}. Champneys and coworkers \cite{Champneys2009} give a detailed geometric analysis of the dynamics near an $EP1t$-cycle, they investigate the existence of one-pulse $E$-homoclinics and $P$-homoclinic tangencies as well as periodic solutions close to the heteroclinic cycle. Moreover, in the presence of non-real eigenvalues, they study homoclinic snaking behaviour and give numerical examples of corresponding systems that arise in applications. A mode of intracellular calcium dynamics has also been studied. It is our aim to contribute to extend these studies.

In recent years the continuation of homoclinic orbits unfolding heteroclinic cycles received a great deal of attention. Motivated by the Swift-Hohenberg equation of subcritical type, a rigorous analysis in the conservative and reversible context including the study near the periodic solution has been established in Beck \emph{et al.}\ \cite{Beck}. 
A discussion of this phenomenon without any additional structure has been performed in Knobloch \emph{et al.}\ \cite{KRV2011}. In \cite{KRV2011}, using the functional analysis Lin's method, the authors prove rigorously some results of \cite{Champneys2009} and study numerically an explicit example. Moreover, we also refer to Rademacher \cite{Rad_PhD} where the author studied the dynamics near singular cycles.

The \emph{Homoclinic Snaking Problem} refers to the snaking continuation curve (in the bifurcation diagram) of homoclinic cycles near an $EP1$-cycle.  This topic has been explored in Knobloch and Rie\ss\ \cite{KnobRiez2010}, where the authors apply Lin's method to a differential equation unfolding a saddle-node Hopf bifurcation with global reinjection introduced in \cite{KrauskopfOldeman}.  Detection and continuation of homoclinic cycles circulating several times around a cycle (later called multipulses) have been studied by  Knobloch \emph{et al.} \cite{Knob_JDDE_2011}. In the latter article, the cycle considered is different from ours in two respects: firstly, the  equilibrium is a saddle-focus and secondly, the authors explicitly use reversibility.  Many problems remain to be solved as we point out at the end of the present work.

\subsection{Framework of the paper}
Our work is structured as follows. In Section \ref{setting} we introduce the general setting in which $EP1t$-cycles arise. Section \ref{local} sets up the local maps around $E$ and $P$, as well as the transition maps between them, using mainly terminology from \cite{Champneys2009}. In Sections \ref{localdyn_E} and \ref{localdyn_P} we characterize the dynamics near $E$ and $P$, leading to the statement and proof of our result on the dynamics of the organizing center in Section \ref{organizing_center}. Section \ref{sec-multipulses} is concerned with multipulse homoclinics. In Section \ref{chaos-dynamics} we give an overview of the (chaotic) dynamics and bifurcations for the different cases we encounter. Section \ref{conclusion} concludes. Throughout this paper, we have endeavoured to make a self contained exposition bringing together all topics related to the proofs. We have stated short lemmas and we have drawn illustrative figures to make the paper easily readable.

\section{The setting}\label{setting}
Let $M$ be a compact and boundaryless three-dimensional manifold and let $\mathcal{X}^\infty(M)$ the Banach space of $C^\infty$ vector fields on $M$ endowed with the $C^\infty$ Whitney topology.

Our object of study is the dynamics around a special type of heteroclinic cycle, an $EP1t$-cycle, which occurs in the Swift-Hohenberg equation, and for which we give a rigorous description here. Specifically we consider  a two-parameter family of vector fields in $\mathcal{X}^\infty(M)$,  with a flow given by the unique solution $x(t)=\varphi(t,x) \in M$ of
\begin{equation}
\label{general}
\dot{x}=f(x, \alpha, \beta), \qquad x(0)=x_0\in M
\end{equation}
where $\alpha$ and $\beta$ are real parameters, and the organizing center $\alpha=\beta=0$ satisfies the following hypotheses:
\medbreak
\begin{enumerate}
\item[\textbf{(H1)}] $E$ is a hyperbolic equilibrium where the eigenvalues of $\mathrm{D}f(x, 0,0)|_{x=E}$ are $\mu,-\lambda_1,-\lambda_2 \in \RR$ and satisfy $-\lambda_2<-\lambda_1<0< \mu$ and $\lambda_1<\mu$.
\medbreak
\item[\textbf{(H2)}] $P$ is a hyperbolic periodic solution of minimal period $T>0$ with positive Floquet multipliers $0<\exp(-\eta_c)<1<\exp(\eta_e)$ for $\eta_e, \eta_c>0$. In addition,  $(-\eta_c, \eta_e ) \in \mathcal{D}$ where  $\mathcal{D}\subset \RR^2$ is the residual set given in \cite{Takens71}.
\medbreak
\item[\textbf{(H3)}] There is a heteroclinic connection $W^u(P) \cap W^s(E)$, given through a quadratic tangency, that we denote by $[P \rightarrow E]_t$.
\medbreak
\item[\textbf{(H4)}] There is a heteroclinic connection $ W^u(E) \cap W^s(P)$ that we denote by $[E \rightarrow P]$.
\end{enumerate}
\medbreak
The equilibrium $E$ is what is widely known as a \emph{Lorenz-like singularity}. By \textbf{(H2)}, points in the set $\mathcal{D}\subset \RR^2$ satisfy a suitable Sternberg $k$-condition, thus excluding resonances, which are phenomena of zero Lebesgue measure. Since we restrict to a residual set of Floquet multipliers, we can say that the assumptions of our results hold generically. More details are given later in Remarks \ref{REM1} and \ref{REM2}.
\medbreak

In \cite{MP}, the $[P \rightarrow E]_t$ connection of \textbf{(H3)} is called $\gamma_0$ and the $[E \rightarrow P]$ connection of \textbf{(H4)} is called $\gamma_1$.  The flow of the organizing center has a singular cycle associated to $E$ and $P$, that we denote by $\Gamma_{(0,0)}$ or simply $\Gamma$ -- see Figure \ref{EP1-cycle}. Its occurrence is a phenomenon of codimension two because of properties \textbf{(H3)} and \textbf{(H4)}. We also assume that: 
\begin{enumerate}
\item[\textbf{(H5)}] The two parameters $(\alpha, \beta)$ satisfy:
\begin{itemize}
\item[\textbf{(a)}] $\alpha$  and $\beta$ act independently;
\item[\textbf{(b)}] $\alpha$ unfolds the codimension one heteroclinic connection from $E$ to $P$;
\item[\textbf{(c)}] $\beta$ unfolds the codimension one tangent heteroclinic connection from $P$ to $E$.
\end{itemize}
\end{enumerate}
\medbreak
In particular, the codimension one connections from $E$ to $P$ occur on the $\beta$-axis, given by $\alpha=0$, and the codimension one tangent heteroclinic connections from $P$ to $E$ occur on the $\alpha$-axis, given by $\beta=0$. We use the notation $f_0(x)=f(x,0,0)\in \mathcal{X}^\infty(M)$ to refer to the vector field of the organizing center.

\begin{figure}
\begin{center}
\includegraphics[height=3cm]{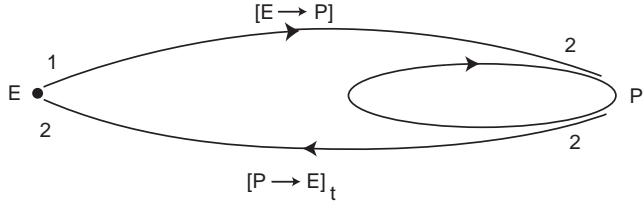}
\end{center}
\caption{\small Schematic representation of an $EP1t$-cycle: the hyperbolic equilibrium $E$ has a one-dimensional unstable manifold and a two-dimensional stable manifold; the two invariant manifolds of the periodic solution $P$ have dimension 2. The unstable manifold of $P$ meets the stable manifold of $E$ tangentially.}
\label{EP1-cycle}
\end{figure}
\medbreak

\begin{definition}
A heteroclinic cycle $\Gamma$ is said to be \emph{isolated} if there is an open set $\mathcal{U}\subset M$ such that $\bigcap_{t \in \RR} \varphi(t,\mathcal{U})=\Gamma$. The set $\mathcal{U}$ is called an \emph{isolating block}.
\end{definition}
 
The study in \cite{MP} focuses exclusively on isolated $EP1t$-cycles, while in \cite{Champneys2009}, the authors consider a greater variety of cases arising through different signs of the occuring parameters. We largely follow the notation in \cite{Champneys2009}. For the most important parameters the corresponding quantities in \cite{Champneys2009} and \cite{MP} are listed in Table \ref{notation} for ease of reference.
\medbreak
In the following, we construct a toy model of the dynamics near such a degenerated heteroclinic cycle in terms of Poincar\'e maps between neighbourhoods of $E$ and $P$, where the flow may be $C^2$-linearized. Then we study geometric properties of these maps and analyse the resulting algebraic bifurcation equations. Appropriate composition of linear and global maps gives the desired first return maps.

\section{Local and Transition Maps}
\label{local}
We study the dynamics near the cycle by deriving local and transition maps that approximate the dynamics close to and between the two saddles in the cycle. Here we establish notation that we use throughout the paper, and the expressions for the local and transition maps. In our analysis we distinguish between local and global invariant manifolds of the saddles: when looking at a cross section, we denote by $W^s_\loc(.)$ the set of points converging directly to the respective saddle, while those in $W^s(.)$ make at least another turn around the other saddle before doing so. Analogously we use $W^u_\loc(.)$ and $W^u(.)$ for the local and global unstable manifolds. 

\subsection{Local map near $E$}
\label{linearization E}
The condition of $C^1$-linearization considered in \cite{Champneys2009} is not adequate for the problem under consideration. For instance, in $C^1$-coordinates, it is impossible to define the type of a tangency (quadratic) that will be considered later; one can only speak about a topological type of the tangency: one-sided or topological intersection.  Therefore, we need to explore a normal form when studying the saddle equilibrium state. 
\medbreak

Denoting by $V_E$ a small cubic neighbourhood of $E$ of size $1$ as illustrated in Figure \ref{local-coord_E} (a), the generalization of Bruno's Theorem \cite[Appendix A]{Shilnikov et al} says that we may choose local cartesian $C^{2}$ coordinates $(x,y,z)$ within $V_E$, so that locally the system (\ref{general}) casts as follows:
\begin{equation}
\label{nf1}
\left\{ 
\begin{array}{l}
\dot x = -\lambda_1 x + \tilde{f}_{11}(x,y,z)x + \tilde{f}_{12}(x,y,z)y, \\
\dot y =  - \lambda_2 y + \tilde{f}_{21}(x,y,z)x + \tilde{f}_{22}(x,y,z)y, \\
\dot z = \mu z + \tilde{g}(x,y,z)z
\end{array}
\right.
\end{equation}
where $\tilde{f}_{11},\tilde{f}_{12},\tilde{f}_{21},\tilde{f}_{22}$ and $\tilde{g}$ are at least $C^{2}$-smooth with respect to $(x,y,z)$ and the following identities are valid for $i,j \in \{1,2$\}:
$$
\tilde{f}_{ij}(0,0,0)=\tilde{f}_{1i}(x,y,0)=\tilde{f}_{j1}(0,0,z)=\tilde{g}(0,0,0)=\tilde{g}(x,y,0)=\tilde{g}(0,0,z)=0.
$$
In
this case, as proved in \cite{Shilnikov et al}, the linear map and the high order terms depend smoothly on the parameters $(\alpha, \beta)$. 

\medbreak
\medbreak

Without loss of generality, we set that the connection $[E \to P]$ leaves $E$ in positive $z$-direction, and the connection $[P \to E]_t$ reaches $E$ along the $x$-axis. After a linear rescaling of the variables, we consider two local sections $\Sigma_E^\inn$ and $\Sigma_E^\out$ across the $x$- and $z$-axis, respectively, such that $x=1$ in $\Sigma_E^\inn$ and $z=1$ in $\Sigma_E^\out$. Depending on the sign of the $z$-component in $\Sigma^{\inn}_E$ we define
\begin{equation*}
 \Sigma^{\inn(+)}_E=\{(y,z) \mid z>0 \}, \quad \Sigma^{\inn(-)}_E=\{(y,z) \mid z<0 \}.
\end{equation*}
Note that $W^s_{\loc}(E)$ corresponds to the plane $z=0$. Up to high order terms the flow $\varphi(t,.)$ for the normal form  (\ref{nf1}) induces a map that sends the $(y,z)$ coordinates in $\Sigma_E^{\inn(+)}$ to the $(x,y)$ coordinates in $\Sigma_E^\out$:
\begin{equation}
\Pi_E:\Sigma^{\inn(+)}_E \rightarrow \Sigma^{\out}_E, \quad \Pi_E(y,z)=\left(z^{\delta_1} + h.o.t., yz^{\delta_2}+ h.o.t. \right) \cong (x,y),
\label{local_E}
\end{equation}
where  $\delta_1=\frac{\lambda_1}{\mu}<1$ and $\delta_2=\frac{\lambda_2}{\mu}$. The \emph{higher order terms (h.o.t.)} in (\ref{local_E}) have been derived in Appendix B of \cite{Shilnikov et al}, by considering integral formulae for solutions obtained using variation of constants. See also \cite[Subsection 3.2]{HS}.
\medbreak
 Observe that $\delta_1<\delta_2$ by \textbf{(H1)}, and if $y \in \RR$ is fixed, then $\lim_{z\rightarrow 0} \Pi_E(y,z)= (0,0)$. By construction, it follows that solutions starting
\begin{enumerate}
\item in $\Sigma^{\inn}_E$ enter $V_E$ in positive time -- those in $\Sigma^{\inn (+)}_E$ then exit through $\Sigma^{\out}_E$ towards $P$, while those in $\Sigma^{\inn(-)}_E$ do not follow the cycle -- see Remark \ref{remark-a21};
\item in $\Sigma^{\out}_E$ exit $V_E$ in positive time.
\end{enumerate}

 \begin{figure}
\begin{center}
\includegraphics[height=8cm]{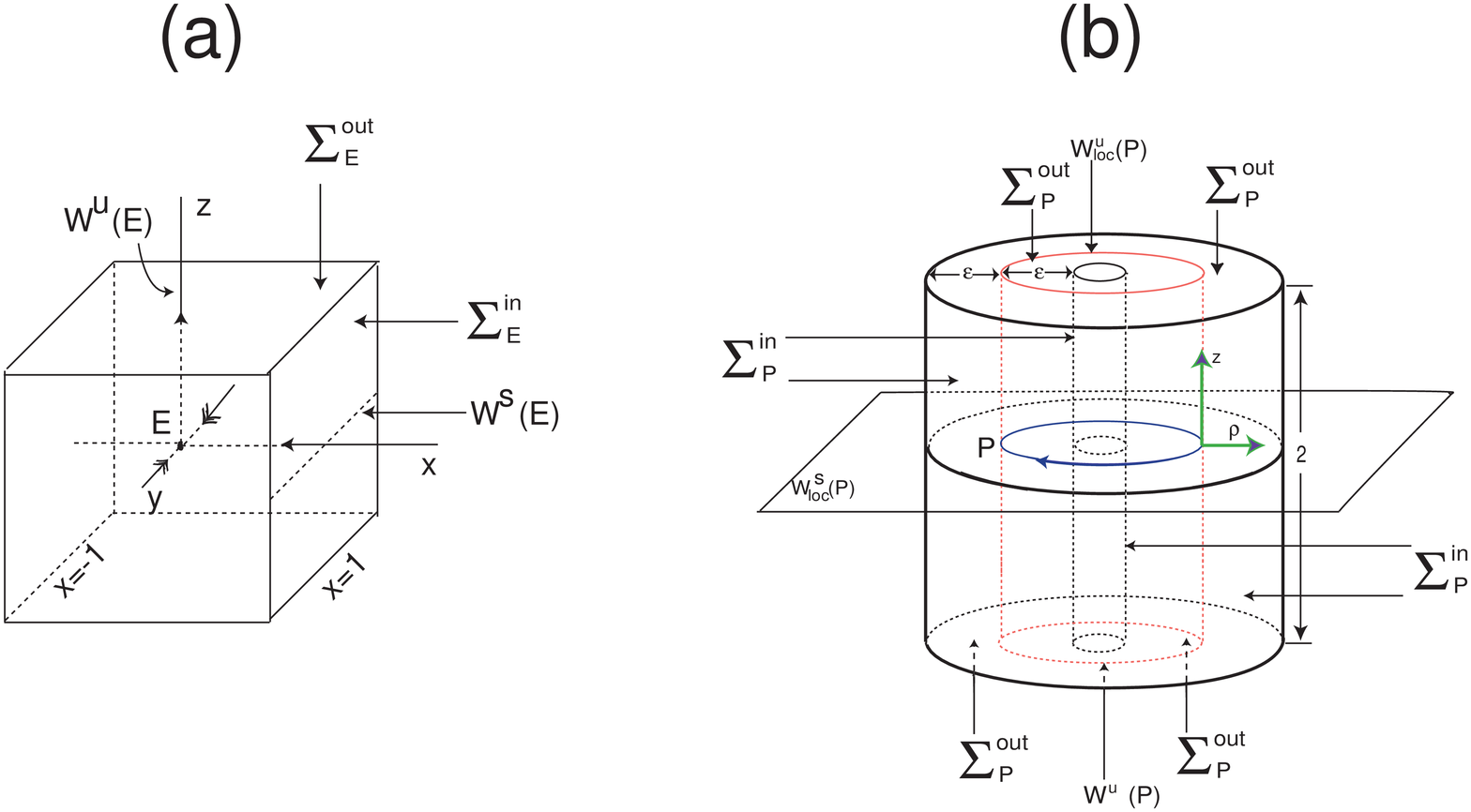}
\end{center}
\caption{\small Local coordinates near the equilibrium $E$ (a) and the periodic solution $P$ (b). For $X\in\{E,P\}$, the flow enters the neighbourhood $V_X$ transversely across the wall $\Sigma_X^{\inn}\backslash W^s_{\loc}(X)$ and leaves it transversely across the top $\Sigma_X^{\out}$ and bottom.}
\label{local-coord_E}
\end{figure}

Denote by $E^{cu}$ the subspace associated to the eigenvalues $\mu>0$ and $-\lambda_1<0$ at $E$. Besides the two-dimensional stable and one-dimensional unstable manifolds of $E$, there exists a centre-unstable manifold of $E$, denoted by $W^{cu}_{loc}(E)$, that is tangent to the subspace $E^{cu}$.  It is clear that $[E \rightarrow P] \subset W^{cu}(E) \cap W^s(P)$.  Although the tangent space $E^{cu}$ at $E$ is unique, the associated center-unstable manifold is not. More details in \cite{Homburg}.

\subsection{Local map near $P$}
\label{PSolution}
Without loss of generality, let us assume that $T=2\pi$.
Let $C$ be a cross section transverse to the flow at $p\in P$. Since $P$ is hyperbolic, there is a neighbourhood $V\subset C$ of $p$ in $C$ where the first return map to $C$ is $C^1$-conjugate to its linear part. In our case, \textbf{(H2)} guarantees that this conjugacy is even $C^2$-smooth, since $f_0 \in \mathcal{X}^\infty(M)$ and $(-\eta_c, \eta_e ) \in \mathcal{D}$ where $\mathcal{D}\subset \RR^2$ is the residual set given in \cite{Takens71}. The arguments and computations do not depend on the specific choice of $p \in P$ and have been explicitly done in  \cite[Appendix]{LR2016}.

\begin{remark}
\label{REM1}
The main result of \cite{Takens71} is proved by using fiber contraction arguments on spaces of $k$-jets ($k\in \NN$) to find coordinates in which the action of the derivative of the diffeomorphism has high-order contact and is in a certain standard form. Based on this information and motivated by \cite{McSwiggen}, we conjecture that the condition about the $C^\infty$ regularity of the initial vector field can be relaxed (as well as the topology).
\end{remark}

\begin{remark}
\label{REM2}
The assumptions on the eigenvalues of the first return map are a variant of those of Sternberg: for any pair $\nu_1, \nu_2$ of non-negative integers satisfying $$2 \leq  \nu_1+\nu_2 \leq \alpha,$$  it is required that the linear combination $\nu_1 \log \lambda_e +\nu_2 \log \lambda_c$ has non-zero real part different from the real part of $\log \lambda_j$, for  $j \in \{e,c\}$. Here, $\alpha$ is an integer that depends on the degree of smoothness required -- details in \cite{LR2016, Takens71}. 
\end{remark}

  Suspending the linear map gives rise, in quasi-cylindrical coordinates $(\psi, \rho, z)$ around $P$, to the system of differential equations:
\begin{equation*}
\left\{ 
\begin{array}{l}
\dot{\psi}= 1\\
\dot{\rho}=-\eta_c\rho  \\ 
\dot{z}=\eta_e z
\end{array}
\right.
\end{equation*}
which is $C^2$-orbitally equivalent to the original flow near $P$. In these coordinates, the periodic solution $P$ is the circle defined by $\rho=0$ and $z=0$, 
its local stable manifold, $W^s_{\loc}(P)$, is  the plane $z=0$ and $W^u_{\loc}(P)$ is the surface defined by $\rho=0$.
After a rescaling of variables if necessary, we work with a hollow three-dimensional cylindrical neighbourhood $V_P$ of $P$ contained in the suspension of $V$:
$$
V_P=\left\{ (\psi, \rho, z):\quad \psi\in\RR\pmod{2\pi}, \quad -1\le\rho\le 1,
\quad -1\le z\le 1 \right\}\  .
$$
Its boundary is a disjoint union of
\begin{itemize}
\item[(i)] the walls of two cylinders ($\rho=\pm 1$), locally separated by $W^u_{\loc}(P)$, and 
\item[(ii)] two annuli, the top and the bottom of the cylinder ($z=\pm 1$),  locally separated by $W^s_{\loc}(P)$.
\end{itemize}
Analogously to the cross sections near $E$ we label the relevant parts of the boundary $\Sigma^{\inn}_P$ (where $\rho=1$) and $\Sigma^{\out}_P$ (where $z=1$). Again we split $\Sigma^\inn_P$ into $\Sigma^{\inn (+)}_P$ and $\Sigma^{\inn (-)}_P$, depending on the sign of $z$. The cylinder wall $\Sigma^{\inn}_P$ is parametrized by the covering map
$$ (\psi,z)\mapsto(\psi,1,z)=(\psi,\rho,z), $$
where $\psi\in \RR\pmod{2\pi}$ and $|z|<1$. 
In these coordinates, $\Sigma^{\inn}_P\cap W^s_{\loc}(P)$ is a circle in the plane where $z=0$. The annulus $\Sigma^{\out}_P$ is parametrized by the covering
$$ (\psi, \rho) \mapsto (\psi, \rho, 1)=(\psi,\rho,z), $$
where $\psi \in \RR\pmod{2\pi}$ and $|\rho|<1$. Then $\Sigma^{\out}_P\cap W^u_{\loc}(P)$ is the circle $\rho=0$. It follows directly by construction that solutions starting
\begin{enumerate}
\item in $\Sigma^{\inn}_P$ enter the hollow cylinder $V_P$ in positive time -- those in $\Sigma_P^{\inn(+)}$ \emph{ie} with $z>0$ exit through $\Sigma_P^{\out}$ towards $E$, while those in $\Sigma_P^{\inn(-)}$ (with $z>0$) do not follow the cycle;
\item in $\Sigma^{\out}_P$ leave the hollow cylinder $V_P$ in positive time.
\end{enumerate} 
The local map near $P$ sends $(\psi,z)$ coordinates in $\Sigma_P^\inn$ to $(\psi, \rho)$ coordinates in $\Sigma_P^\out$. It is then given by
\begin{equation*}
\Pi_{P}: \Sigma^{\inn (+)}_P \to \Sigma^{\out}_P, \quad \Pi _{P}(\psi,z)= \left(\psi-\frac{1}{\eta_e}\ln{z} \mod{2\pi}, {z}^{\delta_P}\right) \cong (\psi, \rho),
\end{equation*}
where $\delta_P=\frac{\eta_c}{\eta_e}>0$ is called the \emph{saddle index} of $P$. A corresponding map can be constructed from $\Sigma^{\inn (-)}_P$ to the annulus where $z=-1$, but we are only interested in following the heteroclinic cycle $\Gamma$, which continues in positive $z$ direction. 

\begin{figure}
\begin{center}
\includegraphics[height=6cm]{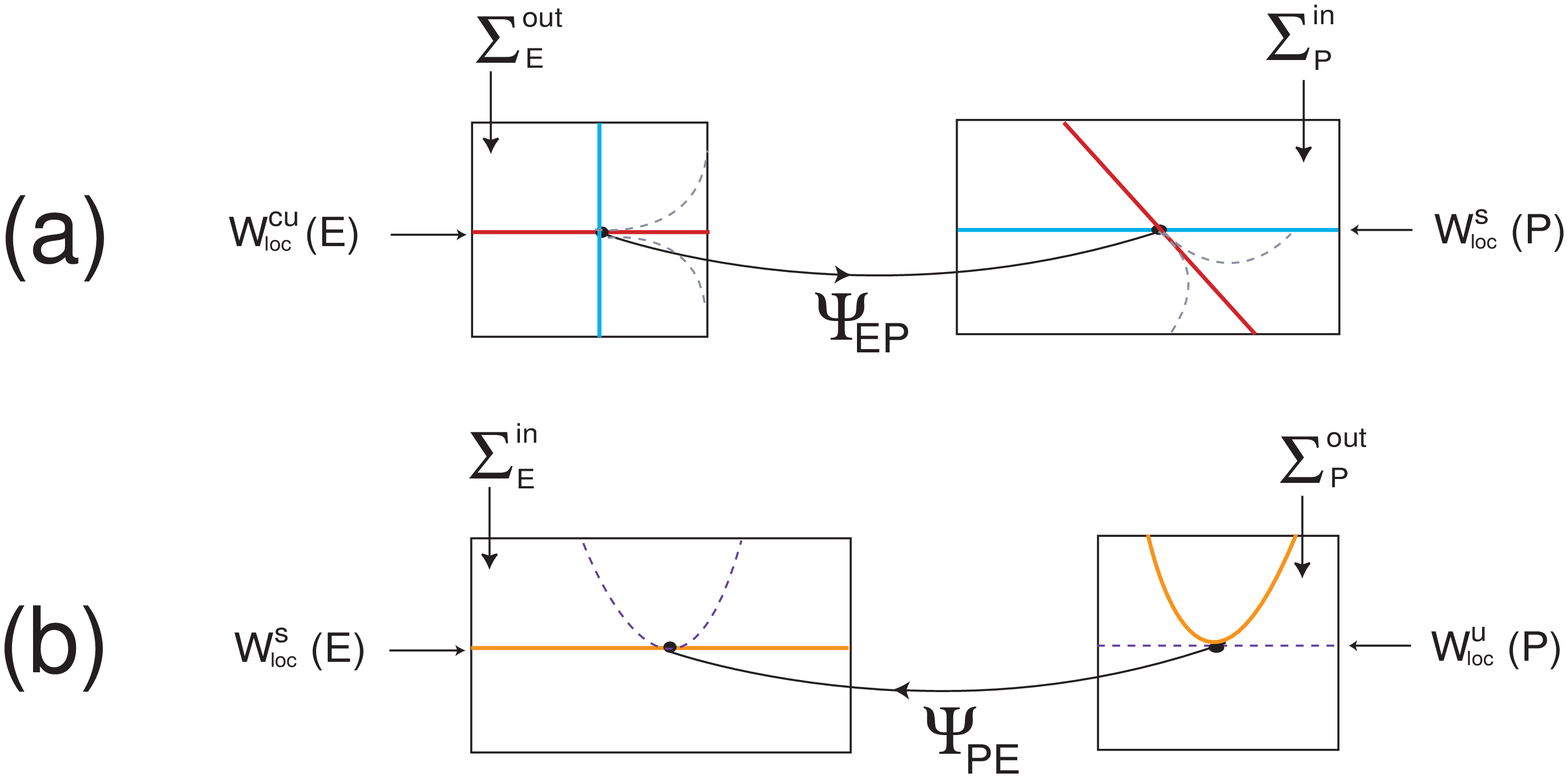}
\end{center}
\caption{\small Scheme of the transition maps for a given choice of parameters ($a_{21}<0$ and $k>0$ and $b_{22}<0$) considered in \cite{MP}. (a) Relative position of $\Psi_{EP}(W^{cu}_\loc(E) \cap \Sigma_E^{\out})$ and $W^s_{\loc} (P)\cap \Sigma_P^\inn$. (b) Orientation of the parabola $\Psi_{PE}(W^u_\loc(P)\cap \Sigma_P^{\out})$ in $\Sigma_E^\inn$ and the orientation-reversibility of $\Psi_{PE}$.}
\label{transitions2}
\end{figure}

\subsection{Transition maps} Retaining just the lowest order terms, the transition map $\Psi_{EP}: \Sigma_E^\out \to \Sigma_P^\inn$ from $E$ to $P$ may be written as
\begin{equation*}
\Psi_{EP} (x,y)=
 \left( 
  \begin{array}{cc}
a_{11} & a_{12}\\
a_{21} & a_{22}\\
\end{array} 
\right)
 \left( 
  \begin{array}{c}
x\\
y\\
\end{array} 
 \right)
+
\alpha 
 \left( 
  \begin{array}{c}
\xi_1\\
1 \\
\end{array} 
\right) \cong (\psi, z),
\end{equation*}
where $\xi_1 \in \RR$, $\det
 \left( 
  \begin{array}{cc}
a_{11} & a_{12}\\
a_{21} & a_{22}\\
\end{array}
\right)
 \neq 0$ and $a_{21} \neq 0$. When $\alpha=0$, then $\Psi_{EP} (0,0)=(0,0)$; this is consistent with the fact that when $\alpha=0$, the connection $[E \rightarrow P]$ persists.  We also assume without loss of generality that $a_{11}>0$ and consider the two cases $a_{21} \gtrless 0$.
  The constant $a_{21}\neq 0$ has a geometric meaning: according to its sign, it indicates the relative position of $\Psi_{EP}(W^{cu}_\loc(E)\cap \Sigma_E^{\out})$ and $W^s_{\loc} (P)\cap \Sigma_P^\inn$. It follows directly that: 

\begin{remark}
\label{remark-a21}
If $a_{21}<0$, then the cycle $\Gamma$ is isolated, because points in $\Sigma_P^\inn$ with $z<0$ do not follow the cycle. This is the case investigated in \cite{MP}. If $a_{21}>0$, other invariant sets near the cycle may occur as we shall see.
\end{remark}
\bigbreak
As before, up to high order terms, the transition map $\Psi_{PE}: \Sigma_P^\out \to \Sigma_E^\inn$ from $P$ to $E$ can be written as
 \begin{equation*}
\Psi_{PE} (\psi,\rho)=
 \left( 
  \begin{array}{cc}
b_{11} & b_{12}\\
0 & b_{22}\\
\end{array} 
\right)
 \left( 
  \begin{array}{c}
\psi\\
\rho\\
\end{array} 
 \right)
+
 \left( 
  \begin{array}{c}
0\\
k\psi^2 \\
\end{array} 
 \right)
+
\beta
 \left( 
  \begin{array}{c}
\xi_2\\
1+ \nu \psi \\
\end{array} 
\right)
\cong (y, z),
\end{equation*}
where $b_{22} \neq 0$ and $\xi_2 \in \RR,\ \nu \in \RR^+$. The manifold $W^u_{\loc}(P)\cap \Sigma_P^{\out}$ is given by $\rho=0$. As stated in \textbf{(H3)}, the manifolds $W^u_{\loc}(P)\cap \Sigma_E^{\inn}$ and $W^s(E)\cap \Sigma_E^{\inn}$ intersect quadratically for $\beta=0$. In fact, the set $W^u_{\loc}(P)\cap \Sigma_P^{\out}$ is mapped under $\Psi_{PE}$ into the of parabola given by:
$$ y=b_{11}\psi \quad \text{and} \quad z=k\psi^2. $$
\medbreak
The tangency is a phenomenon of codimension 1. For constant sign of $\beta$, the parameter $k$ denotes where the transverse heteroclinic cycle occurs and $\beta$ unfolds the tangency in this way: 
\begin{enumerate}
\item[(i)]for $\beta=0$, there is a tangency between $W^u_{\loc}(P)$ and $W^s_{\loc}(E)$;
\item[(ii)] for $k<0<\beta$ or $\beta<0<k$, the tangency evolves into two transverse connections from $P$ to $E$;
\item[(iii)] for $k,\beta<0$ or $k,\beta>0$, there is no connection from $P$ to $E$.
\end{enumerate}
\bigbreak
The case $k>0$  corresponds to the inward case of \cite{MP}; in terms of the bifurcation diagram in $(\alpha, \beta)$, the case $k<0$ is obtained by reflection along the $\alpha$-axis. We look at $b_{22} \gtrless 0$ and $k \gtrless 0$. Summing up we get the following set of cross sections and local/global maps that approximate the dynamics in a neighbourhood of the cycle:
\bigbreak
\hspace*{4.5cm}
\begindc{\commdiag}
\obj(0,0){$(x,y) \in \Sigma_E^{out}$}
\obj(40,0){$(\psi,z) \in \Sigma_P^{in}$}
\obj(40,-20){$(\psi,\rho) \in \Sigma_P^{out}$}
\obj(0,-20){$(y,z) \in \Sigma_E^{in}$}
\mor{$(x,y) \in \Sigma_E^{out}$}{$(\psi,z) \in \Sigma_P^{in}$}{$\Psi_{EP}$}[\atright,\aplicationarrow]
\mor{$(\psi,z) \in \Sigma_P^{in}$}{$(\psi,\rho) \in \Sigma_P^{out}$}{$\Pi_{P}$}[\atright,\aplicationarrow]
\mor{$(\psi,\rho) \in \Sigma_P^{out}$}{$(y,z) \in \Sigma_E^{in}$}{$\Psi_{PE}$}[\atright,\aplicationarrow]
\mor{$(y,z) \in \Sigma_E^{in}$}{$(x,y) \in \Sigma_E^{out}$}{$\Pi_{E}$}[\atright,\aplicationarrow]
\enddc
\bigbreak

For $(\alpha, \beta)=(0,0)$, without loss of generality, we choose the local coordinates such that the connections pass through each cross section in $(0,0)$. In particular, this means that, up to multiples of $2\pi$, the angular component of the two points where the heteroclinic connections meet $\Sigma_P^{\inn}$ and $\Sigma_P^{\out}$ is $0$. See Figure \ref{transitions2} for an illustration of $\Psi_{EP}$ and $\Psi_{PE}$.
\medbreak
Points with negative $z$-component in $\Sigma_P^\inn$ and $\Sigma_E^\inn$ do not follow the cycle and are therefore disregarded in the following considerations. Eight cases of interest arise in this construction, given through $a_{21} \gtrless 0$, $b_{22} \gtrless 0$ and $k \gtrless 0$. The meaning of the constants $a_{21}$, $k$ and $b_{22}$ is indicated in Table \ref{notation}. Note that the authors of \cite{MP} assume $a_{21}<0$. We are mainly interested in $a_{21}>0$. For the sake of simplicity we often treat $\Psi_{EP}$ as a rotation of coordinates by $\pi/2$, an assumption that does not affect the validity of the arguments.

\begin{table}[ht]
\begin{center}
\begin{tabular}{cccl}
Parameter & In \cite{Champneys2009} & In \cite{MP} & Function (according to their sign)  \\ \hline
$a_{21}$ & $c$ & $\partial_{z_0}B(0,0,X)$ & \small{Relative position of $\Psi_{EP}(W^{cu}_\loc(E)\cap \Sigma_E^{\out})$ and $W^s_{\loc} (P)\cap \Sigma_P^\inn$} \\ \hline
$k$ & $k$ & $\partial_{yy}\phi(0,1,X)$ & \small{Orientation of the parabola $\Psi_{PE}(W^u_\loc(P) \cap \Sigma_P^{\out})$ in $\Sigma_E^\inn$:} \\
& & & \small{inward ($k>0$) or outward ($k<0$)}\\ \hline
$b_{22}$ & $s$ & $\partial_yA(0,1,X)$ & \small{Orientation of $\Psi_{PE}$:} \\
& & & \small{inclination flip ($b_{22}<0$) or not ($b_{22}>0$)} \\ \hline
\hline
\end{tabular}
\end{center}
\caption{Parameters $a_{21}$, $k$ and $b_{22}$: corresponding notation in \cite{Champneys2009} and \cite{MP} and their role.}
\label{notation}
\end{table}

\section{Local Dynamics near the equilibrium $E$}\label{localdyn_E}
In this section we use the coordinates and notation introduced in the previous section to study the local geometry near the equilibrium $E$. In particular, we state two results that determine how curves with certain properties change their shape when mapped from $\Sigma_E^\inn$ to $\Sigma_E^\out$.
\begin{definition}
\label{def_segmento} A \emph{ segment }$\gamma$ in  $\Sigma_E^{\inn}$ is a smooth regular parametrized curve $\gamma :[0,1]\rightarrow \Sigma_E^{\inn}$ such that:
\begin{itemize}
 \item[(i)] it meets $W^{s}_{\loc}(E)$ transversely at the point $\gamma (0)$ only;
 \item[(ii)] its velocity does not vanish, \emph{i.e.}\ $\|\gamma'(t)\| \neq0$ for all $t \in[0,1]$;
 \item[(iii)] writing $\gamma (s)=(y(s),z(s))$, both $y$ and $z$ are monotonic functions of $s$.
\end{itemize}
\end{definition}
Analogously, adapting the coordinates, we define a segment in  $\Sigma_P^{\inn}$. Note that if $\gamma $ is a segment in  $\Sigma_E^{\inn}$ (or $\Sigma_P^{\inn}$), then $\gamma'(0)$ is not collinear with $(1,0)$.
It is easy to see that $\Pi_E(0,z)=(z^{\delta_1},0)$. By an abuse of terminology, from now on, let us denote by $E^{cu}$ the center-unstable eigenspace at $E$ projected onto $\Sigma_E^\out$. The next result allows us to characterize of the local dynamics near $E$.
 
\begin{lemma}\label{lemma-segment}
Let $\gamma $ be a segment in $\Sigma_E^{\inn}$, parametrized by $t \in (0,1]$. Then $\Pi_E(\gamma(t))$ is a curve in $\Sigma_E^{\out}$ such that $\lim\limits_{t\rightarrow 0^+} \frac{\Pi_E'(\gamma(t))}{\|\Pi_E'(\gamma(t))\|} \in E^{cu}$.
\end{lemma}
The lemma may be generalized for any curve (not necessarily a segment) that does not lie within $W^s_{\loc}(E)$. The high order terms appearing in (\ref{local_E}) are not relevant in the local geometry near $E$.

\begin{proof}
For $t\in [0,1]$, let $\gamma(t)=(y(t), z(t)) \in \Sigma_E^{\inn}$, where $z(0)=0$ and $z'(0) \neq 0$. Clearly, $\lim_{t\rightarrow 0^+} \Pi_E(\gamma(t))=(0,0)$, since $\gamma(0)=(y(0),0) \in W^s_\loc(E)$. Up to high order terms, by equation (\ref{local_E}) we have $$\Pi_E(\gamma(t))= (z(t)^{\delta_1}, y(t)z(t)^{\delta_2}),$$ which implies
$$
\Pi_E'(\gamma(t))= \left( \delta_1z(t)^{\delta_1-1}z'(t), \ y'(t)z(t)^{\delta_2}+\delta_2y(t)z(t)^{\delta_2-1}z'(t) \right) ,
$$
and thus
$$
\lim_{t \rightarrow 0}\frac{y'(t)z(t)^{\delta_2}+\delta_2y(t)z(t)^{\delta_2-1}z'(t)}{\delta_1z(t)^{\delta_1-1}z'(t)}=\lim_{t \rightarrow 0}\frac{y'(t)z(t)^{\delta_2-\delta_1+1}+\delta_2y(t)z(t)^{\delta_2-\delta_1}z'(t)}{\delta_1z'(t)}=0,
$$ 
because $\delta_2>\delta_1$. Therefore, in the limit $t \to 0^+$, the normalized tangent vector $\frac{\Pi_E'(\gamma(t))}{\|\Pi_E'(\gamma(t))\|}$ points in the direction of the $x$-axis in $\Sigma_E^\out$ and the result follows.
\end{proof}

Let $a, b\in \RR$ such that $b-a <2$. Let $\mathcal{P}_\beta^-$ be a curve in $\Sigma_E^{\inn}$ parametrized by the graph of a quadratic function $h_\beta:[a,b] \rightarrow \RR$, $\beta \in \RR$, such that:
\begin{itemize}
\item $\forall t \in (a,b),$ $h_\beta''(t)<0$;
\item for $\beta>0$, $h_\beta$ has two zeros, say $\omega_1\leq0\leq\omega_2$ and for $\beta<0$, $h_\beta$ does not have zeros and is strictly negative;
\item $h_0\left(y^M_0\right)=0$.
\end{itemize}
Let $M_\beta$ be the unique global maximum of the curve $h_\beta$ attained at $y^M_\beta \in (a,b)$. Analogously, let $\mathcal{P}_\beta^+$ be a curve on $\Sigma_E^{\inn}$ parametrized by the graph of a quadratic function $h_\beta:[a,b] \rightarrow \RR$ such that:
\begin{itemize}
\item $\forall t \in (a,b),$ $h_\beta''(t)>0$;
\item for $\beta<0$, $h_\beta$ has two zeros, say $\omega_1\leq0\leq\omega_2$ and for $\beta>0$, $h_\beta$ does not have zeros and is strictly positive;
\item $h_0\left(y^m_0\right)=0$.
\end{itemize}
Let $m_\beta$ be the unique global minimum of the curve $h_\beta$ attained at $y^m_\beta \in (a,b)$. For the sake of simplicity, assume that $y^M_\beta=0$ and $y^m_\beta=0$. In what follows we need a definition of \emph{cusp} and \emph{cuspidal region}. In the literature (see Castro and Lohse \cite{CastroLohse2015} or Munkres \cite{Munkres}, for instance), we may find slightly varying definitions of these geometrical sets, but for our purposes the following is adequate.

\begin{definition} 
\label{cusp_def}
Let $H \subset M$ be a surface, $p\in H$ and $\ell\subset H$ a line containing $p$. Let $\gamma_1$ and $\gamma_2$ be two $C^1$ curves in $H$, parametrized by $t\in \RR$ such that:
\begin{itemize}
\item $\gamma_1(0)=\gamma_2(0)=p$;
\item $\gamma_1$ and $\gamma_2$ are tangent to $\ell$ only at $p$;
\item each of the traces of $\gamma_{1}$ and $\gamma_{2}$,  for $t\neq 0$, is contained in one semi-plane defined by $\ell$.
\end{itemize}
Each non-convex and asymptotically small region in $H$ bounded by the two curves near $p$ is a \emph{cusp} tangent to $\ell$ and centered at $p$ and its boundary is called a \emph{cuspidal curve} tangent to $\ell$ and centered at $p$. 
\end{definition}
{Let $\mathcal{C}$ a cuspidal curve tangent to $\ell$ and centered at $p$. In a given neighbourhood of $p$, the set $\mathcal{C}\backslash \{p\}$ may either connected or disconnected. }

\begin{definition}
 Let $H \subset M$  be a surface and $\gamma=(\gamma_1,\gamma_2)$ a $C^1$ curve in $H$ parametrized by $t\in \RR$ and let $\gamma(0)=p\in H$. We say that $\gamma$ has a fold point at $p$ if $\gamma_1'(0)=0$, $\gamma_2'(0)\neq 0$ and $\gamma_1'(t)$ changes sign at $t=0$.
\end{definition}

All the cusps and cuspidal regions of the next result will be centered at the point $\Sigma_E^{\out}\cap W^u_{\loc}(E) = \{(0,0)\}$. If $A\subset M$, let $\overline{A}$ denote the topological closure of $A$. 
\begin{proposition}
\label{dyn_E_proposition}
The following statements hold.
\begin{enumerate}
\item If $\beta\leq 0$, then $\Pi_E\left(\mathcal{P}_\beta^-(t)  \right) \cap \Sigma_E^{\out(+)}=\emptyset$.
\item If $\beta>0$, then $\overline{\Pi_E(\mathcal{P}_\beta^-(t))}\cap \Sigma_E^{\out(+)}$ is a connected cuspidal curve tangent to $E^{cu}$ with a fold point at $\left(h_\beta(0)^{\delta_1},0\right)=\left(M_\beta^{\delta_1},0\right)$.
\item If $\beta\leq 0$, then $\overline{\Pi_E(\mathcal{P}_\beta^+(t))}\cap \Sigma_E^{\out(+)}$ is a disconnected cuspidal curve tangent to $E^{cu}$.
\item If $\beta> 0$, then $\overline{\Pi_E(\mathcal{P}_\beta^+(t))}\cap \Sigma_E^{\out(+)} $ is a curve with a fold point at $\left(h_\beta(0)^{\delta_1},0\right)=\left(m_\beta^{\delta_1},0\right)$. The derivative of this curve at the fold point is well defined.
  
\end{enumerate}
\end{proposition}
\begin{figure}[ht]
\begin{center}
\includegraphics[width=\textwidth]{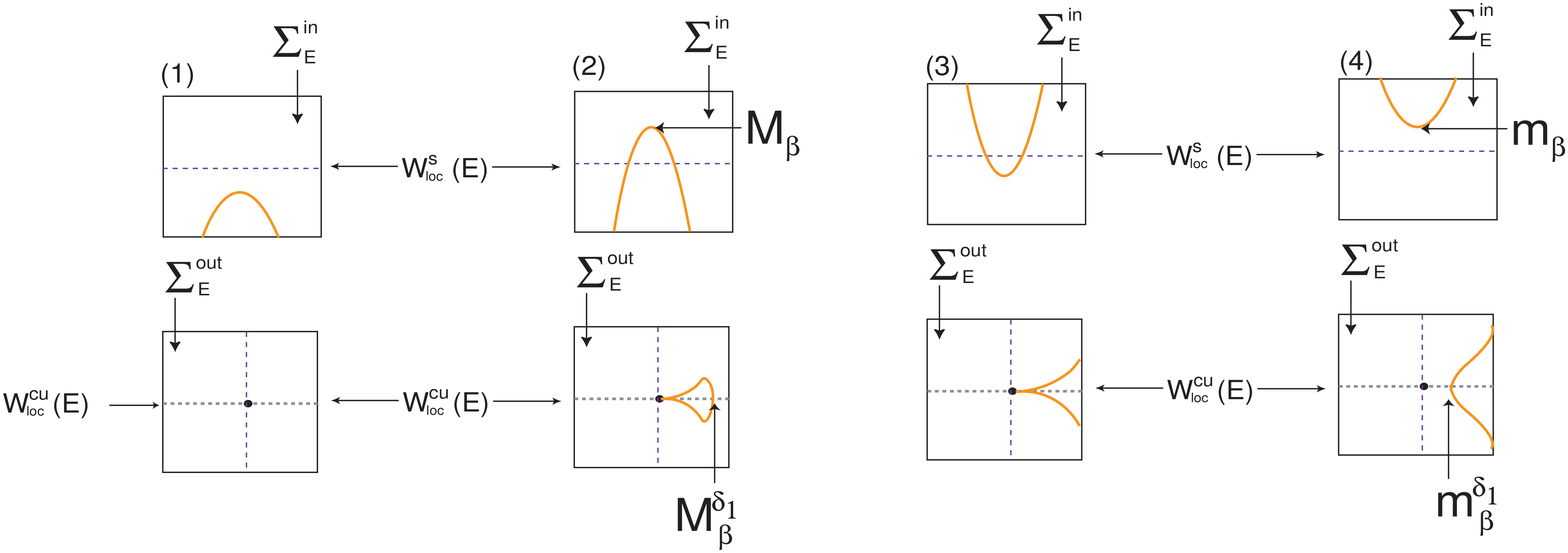}
\end{center}
\caption{\small Illustration of Proposition \ref{dyn_E_proposition}: different images under $\Pi_E$ of a curve of the type $\mathcal{P}_\beta^\pm$.}
\label{Prop5.7}
\end{figure}

\begin{proof}
We suggest that the read follows the following proof observing Figure \ref{Prop5.7}. 
The first statement is clear because $h_\beta(t)\leq0$ for all $t \in [a,b]$, but $z>0$ for all $(y,z) \in \Sigma_E^{\inn(+)}$, which is the domain of $\Pi_E$. The second and third follow by applying Lemma \ref{lemma-segment}, since the part of $\mathcal{P}^\pm$ in $\Sigma_E^{\inn(+)}$ consists of two segments. Indeed,
$$\lim_{t \rightarrow \omega_1^+} \frac{\Pi_E(\mathcal{P}^-_\beta(t))}{\partial t}= \lim_{t \rightarrow \omega_2^-} \frac{\Pi_E(\mathcal{P}^-_\beta(t))}{\partial t} \in E^{cu}.$$
 Note that in (2) the curve is connected because $\{h_\beta(t)>0 \mid t \in [a,b] \}$ is connected and $\Pi_E:\Sigma_E^{\inn(+)} \rightarrow \Sigma_E^{\out(+)}$ is smooth. In (4), $\mathcal{P}^+$ does not intersect $W_\loc^s(E)$, so $\overline{\Pi_E(\mathcal{P}_\beta^+(t))}$ is bounded away from $(0,0) \in \Sigma_E^\out$. Item (4) follows easily that a fold point occurs where $h_\beta$ attains its minimum.
\end{proof}

\section{Local Dynamics near the periodic solution $P$}\label{localdyn_P}

In this section we study the dynamics near $P$. In particular, we characterize geometrically how certain sets are transformed by the local map $\Pi_P$. To this end, we first recall some terminology about horizontal and vertical strips used in \cite{Deng, GH, Rodrigues2_2013} adapted to our purposes. 

\subsection{Strips across $\Sigma_P^{\inn}$}

Let $\psi_1,\psi_2 \in \RR$ such that $\psi_1<\psi_2$. As depicted in Figure \ref{strips}(a), given a rectangular region $\mathcal{R} = [\psi_1,\psi_2]\times [z_1, z_2]$ in $\Sigma_P^{\inn}$, a \emph{horizontal strip} in $\mathcal{R}$ across $\Sigma_P^{\inn}$ is parametrized by:
$$
\mathcal{H}=\{(\psi,z) \in \mathcal{R}: \psi \in [\psi_1,\psi_2], z\in[u_1(\psi),u_2(\psi)]\},
$$
where $ u_1,u_2: [\psi_1,\psi_2] \rightarrow [z_1,z_2]$  are Lipschitz functions such that $u_1(\psi)<u_2(\psi)$. The \emph{horizontal boundaries} of the strip are the lines parametrized by the graphs of $u_i$, $i\in \{1,2\}$, the \emph{vertical boundaries} are the lines  $\{\psi_i\}\times  [u_1(\psi_i),u_2(\psi_i)]$.
The \emph{maximum height} and the \emph{minimum height} are given, respectively, by:
$$
h^{\Max}=\max_{\psi \in [\psi_1,\psi_2]}u_2(\psi) \qquad \text{and} \qquad h^{\min}=\min_{\psi \in [\psi_1,\psi_2]}u_1(\psi). 
$$
The diameter of $\mathcal{H}$, denoted by $\diam(\mathcal{H})$, is given by $\max_{\psi \in [\psi_1,\psi_2]} |u_1(\psi)-u_2(\psi)|\geq 0$. This number exists because $u_1-u_2$ is a continuous map defined on a compact set.

 A \emph{vertical strip} $\mathcal{R}$ has similar definitions as shown in Figure \ref{strips}(b), with the roles of $\psi$ and $z$ reversed. A region $V$ is called a \emph{vertical cusp} across $\Sigma_P^{\inn}$ if 
$$
V=\{(\psi, z)\in \mathcal{R}: z \in [z_1,z_2], \psi \in[v_1(z),v_2(z)] \}
$$
where $ v_1,v_2: [z_1,z_2] \rightarrow [\psi_1,\psi_2]$ are Lipschitz functions such that for all $z\in [z_1, z_2)$, we have $v_1(z)< v_2(z)$ and $v_1(z_1)= v_2(z_1)$ -- see Figure \ref{strips}(c). The point $p^\star=(v_1(z_1), z_1)$ is called the \emph{cusp point}. 
The vertical boundaries of the vertical cusp should be seen as cuspidal curves as in Definition \ref{cusp_def}.

\begin{figure}[ht]
\begin{center}
\includegraphics[width=\textwidth]{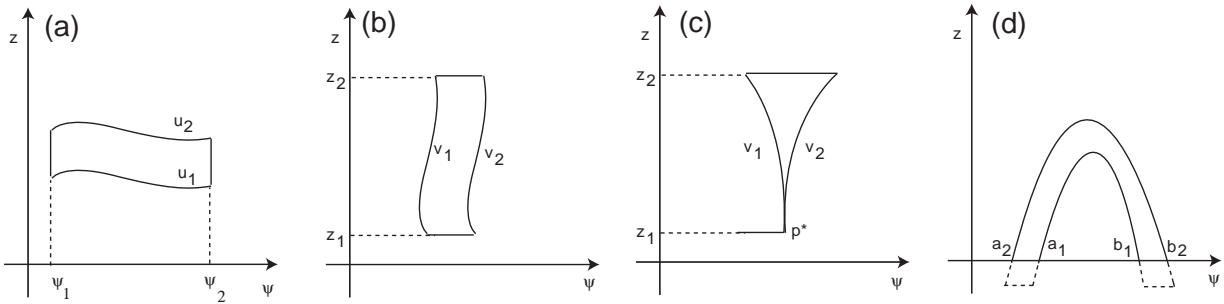}
\end{center}
\caption{\small Different types of strips. (a): Horizontal strip. (b): Vertical strip. (c): Vertical cusp. (d): Horseshoe strip.}
\label{strips}
\end{figure}

As illustrated in Figure \ref{strips}(d), a \emph{horseshoe strip} in $\Sigma_P^{\inn}$ is a subset of $\Sigma_P^{\inn}$ of the form
$$
\mathcal{H}=\left\{ (\psi,z)\in \Sigma_P^{\inn}:\quad \psi\in[a_2,b_2], \quad z \in [u_1(\psi), u_2(\psi)]\right\},
 $$
 where
 \begin{itemize}
\item $[a_1, b_1] \subset [a_2, b_2] \subset [-\tau, \tau]$;
\item $u_1(a_1)=u_1(b_1)=0= u_2(a_2)=u_2(b_2)$.
\end{itemize}
The boundary of $\mathcal{H}$ consists of the graph of $u_2(\psi)$, $\psi\in[a_2,b_2]$, the graph of $u_1(\psi)$,  $\psi\in[a_1,b_1]$, together with the two segments $[a_2, a_1]\times \{0\}$ and $[b_1, b_2]\times \{0\}$. More details in Labouriau and Rodrigues \cite[Section 6]{LR2015}. Similarly, we define horizontal strip,  vertical strip and vertical cusp across $\Sigma_P^{\out}$, $\Sigma_E^{\inn}$ and $\Sigma_E^{\out}$. 

\subsection{Infinite sequence of strips}

In order to describe the dynamics near $P$ we need the geometric notion of a helix. This allows us to characterize the way in which segments in $\Sigma_P^\inn$ get mapped into $\Sigma_P^\out$, and to describe the pre-image of the stable manifold $W_\loc^s(E)$ when pulled back around $P$.
\begin{definition}\label{def_helix} 
Let $a,b\in \RR$ such that $a<b$ and let $H \subset M$ be a surface parametrized by a cover $(\psi,h )\in  \RR\times[a,b]$ where $\psi $ is periodic. A \emph{ helix $\gamma$ on $H$ accumulating on the circle} $h=h_{0}$ is a curve $\gamma :(0,1]\rightarrow H$ such that its coordinates $(\psi (s),h(s))$ are quasi-monotonic functions of $s$ with $$
 \lim_{s\to 0^+}h(s)=h_{0}
 \qquad \mbox{and}\qquad
 \lim_{s\to 0^+}|\psi (s)|=+\infty .
$$
A \emph{double helix on $H$ accumulating on the circle} $h=h_0$ is the union of two disjoint helices $\gamma_1, \gamma_2$ accumulating on $h=h_0$ and a curve connecting their end points $\gamma_1(1)$ and $\gamma_2(1)$.
\end{definition}

For $\alpha=\beta=0$, let $R_P^{\inn}\subset \Sigma_P^{\inn}$ and $R_P^{\out}\subset \Sigma_P^{\out}$ be two squares parametrized by $[-\tau, \tau]\times [-\tau, \tau]$, where $\tau$ is a  small positive number. These squares are centered at the point where the heteroclinic connection meets the cross sections. The following proof is based on \cite{Rodrigues2_2013}.

\begin{lemma}
\label{main_prop}
Let $\gamma$ be a vertical segment in $R_P^{\inn}$ whose angular component is $\psi_0 \in \RR \pmod{2\pi}$. Then there are $n_0 \in \NN$ and a family of intervals $(\mathcal{I}_n)_{n \geq n_0}= \{[\exp(a_n), \exp(b_n)]\}_{n \geq n_0} $, where
$$a_n= \eta_e (-\tau-2n\pi +\psi_0),  \quad \quad b_n= \eta_e ( \tau-2n\pi +\psi_0)$$
such that $\{\psi_0 \} \times [\exp(a_n), \exp(b_n)] \subset \textnormal{graph}(\gamma)$ and $\Pi_P(\{\psi_0\}\times [\exp(a_n), \exp(b_n)])\subset R_P^{\out}.$
\end{lemma}

\begin{proof}
Write $\gamma(s) = (\psi_0, z(s)) \in R_P^{\inn} \subset \Sigma_P^{\inn}$, where $z(s)\geq 0$ is an increasing map of $s$ and $\lim_{s \rightarrow 0^+}z(s)=0$. The function $\Pi_P$ maps the curve $\gamma$ into a helix accumulating on $W^u_{\loc}(P)$ parametrized by:
$$
\Pi_P(\gamma(s))= \Pi_P\left(\psi_0, z(s)\right)= \left(\psi_0-\frac{1}{\eta_e} \ln z(s),z(s)^{\delta_P}\right) \cong (\psi(s), \rho(s)).
$$
Indeed $\rho(s)$ and $\psi(s)$ are monotonic (since $z(s)$ is monotonic)  and
$$
\lim_{s \rightarrow 0^+} z(s)^{\delta_P}=0 \quad \text{and} \quad \lim_{s \rightarrow 0^+} -\frac{\ln(z(s))}{\eta_e}+\psi_0=+\infty.
$$
The sequences $\exp(a_n)$ and $\exp(b_n)$ defining the family of intervals are obtained from points where the helix $\Pi_P(\gamma(s))$ meets the vertical boundaries of $R_P^{\out}$, defined locally by $$(\psi, \rho)=\{\pm \tau + 2 n \pi \}_{n \in \NN}  \times [0,\tau].$$ Let $n_0$ the smallest natural number for which the equation $\psi(s) = \pm \tau+ 2 n \pi$ can be solved for $z$. Then for $n \geq n_0$ we get $z= \exp({a_n})$ and $z= \exp(b_n)$, where
$$
a_n= \eta_e (-\tau -2n\pi +\psi_0) \qquad \text{and} \qquad b_n=  \eta_e (\tau -2n\pi +\psi_0).$$
\end{proof}

It is worth observing that $a_n$ and $b_n$ depend continuously on $\psi_0 \in [-\tau ,\tau]$ and 
\begin{enumerate}
\item $\forall n \geq n_0,$ we have that $1> \exp(b_n)>\exp(a_n)>\exp(b_{n+1})$ and
\item   $\lim_{n \to \infty} \exp(a_n)=\lim_{n \to \infty} \exp(b_n)=0$.
\end{enumerate}
\medbreak

Since $a_n$ and $b_n$ of Lemma \ref{main_prop} depend smoothly on the angular coordinate $\psi_0$ of the vertical segment $\gamma$, the horizontal strips
\begin{equation}
\label{Hn_def}
H_n = [-\tau, \tau] \times  [\exp (a_n(\psi)), \exp (b_n(\psi))] \subset R_P^{\inn},  \quad n \geq n_0 \in \textbf{N}, \quad \psi \in[-\tau, \tau]
\end{equation}
are mapped by $\Pi_P$ into $R_P^{\out}$. The image under $\Pi_P$ of the endpoints of the horizontal boundaries of $H_n$ must join the end points of $\Pi_P(\{\psi_0\}\times [\exp(a_n), \exp(b_n)])\subset R_P^{\out}$. Moreover, for each $n>n_0$, if $h_n^{\Max}$ and $h_n^{\min}$ denote, respectively, the maximum and the minimum height of the horizontal strip $H_n$ then $$
\lim_{n \to \infty} h_n^{\Max}  =\lim_{n \to \infty} h_n^{\min} = \lim_{n \to \infty} h_n^{\Max}-h_n^{\min}  =0.$$
\bigbreak
 \begin{remark}
 \label{boundaries}
 From the previous analysis, it is easy to see that for $n \geq n_0 \in \NN$, the horizontal (resp.\ vertical) boundaries of $H_n$ are mapped into the vertical (resp. horizontal) boundaries of $\widetilde{H}_n=\Pi_P(H_n)$. The sequence $(H_n)_{n\geq n_0}$ converges, in the Hausdorff topology, to the stable manifold of $P$, and $(\widetilde{H}_n)_{n\geq n_0}$ converges to the unstable manifold of $P$.
\end{remark}

\begin{lemma}\label{doublehelix}
\label{helix}
In the following cases:
\begin{enumerate}
\item $b_{22}, k<0$ and $\beta>0$,
\item $b_{22}, k>0$ and $\beta<0$,
\end{enumerate}
the image of the curve $W_\loc^s(E) \cap \Sigma_E^{\inn(+)}$ under  $(\Psi_{PE}\circ\Pi_P)^{-1}$ is a double helix in $\Sigma_P^\inn$ accumulating on $W^s_{\loc}(P)\cap \Sigma_P^\inn$ and connected by a fold point.
 \end{lemma}
\begin{proof}
In both cases, the image of $W_\loc^s(E) \cap \Sigma_E^{\inn(+)}$ under $\Psi_{PE}^{-1}$ consists of two connected vertical segments in $\Sigma_P^{\out}$. Using an argument similar to that of the proof of Lemma \ref{main_prop} and the expression
\begin{equation*}
 \Pi^{-1}_P(\psi, \rho)=\left(\psi+\frac{1}{\eta_e \delta_P}\ln \rho, \rho^{\frac{1}{\delta_P}}\right) \cong (\psi,z) \in \Sigma_P^\inn,
 \end{equation*}
we conclude that the image under $\Pi_P^{-1}$ of each one of these segments is a helix accumulating on $W^s_{\loc}(P)$. Since the transition maps are smooth, the pre-image of the two connected segments in $\Sigma_P^\out$ is a double helix; in particular, the end points are smoothly connected.
\end{proof}

\section{Dynamics for the organizing center}\label{organizing_center}
In this section we put together the local dynamics around the saddles to characterize the dynamics for $\dot{x}=f_0(x)$, the organizing center. In what follows denote by $R$ the map $\Psi_{EP}\circ \Pi_E \circ \Psi_{PE} \circ \Pi_P$ defined on the subset of $$\bigcup_{n\geq n_0} {H}_n \subset R_P^{\inn},\qquad n_0\in \NN$$ that returns to $R_P^{\inn}$.
The letter $\mathcal{U}$ denotes a small neighbourhood of the singular cycle $\Gamma$ such that
$$R_P^{\inn}\subset (\mathcal{U}\cap \Sigma_P^{\inn}) \qquad \text{and} \qquad R_P^{\out}\subset (\mathcal{U}\cap \Sigma_P^{\out}).$$
Recall the definitions of $R_P^{\inn}$ and $R_P^{\out}$ given in Section \ref{localdyn_P} and the expression for $H_n$ given in (\ref{Hn_def}).
Special attention will be given to the cases where we observe chaos (shift dynamics). 

\begin{definition}
We say that there exist \emph{shift dynamics} near the cycle $\Gamma$ if the return map $R$ is topologically conjugate to a shift of at least two symbols. 
\end{definition}
We evoke the theory developed by Deng \cite{Deng} and Moser \cite{Moser} to prove the existence of shift dynamics for the organizing center. The ideas of \cite{Moser} are summarized and well exposed by Wiggins \cite{Wiggins}. Roughly speaking, the works \cite{Moser, Wiggins} give verifiable conditions, sometimes called \emph{Conley-Moser conditions}, for a map to possess an invariant set on which it is topologically conjugate to a Bernoulli shift with at least two symbols. The paradigmatic case where their results may be applied is the classic Smale horseshoe.
\medbreak
Deng's idea \cite{Deng} is to extend the work \cite{Moser} that states that the dynamics of the classic Smale horseshoe may be completely described by the Bernoulli shift automorphism acting on the space of bi-infinite sequence of symbols. The extension \cite{Deng} uses a quotient symbolic system, and thus a semi-conjugacy, which implies that the invariant manifolds of the saddles are dense in the set of trajectories that stay near the singular cycle $\Gamma$. 
 
\subsection{Types of horseshoes}
The following definitions are due to \cite{Deng}. See also \cite[Section 2.3]{Wiggins}.
\begin{definition}\label{reg_horseshoe_def}
Let $H_1$ and $H_2$ be two disjoint horizontal strips across $\Sigma_P^{\inn}$ with $H_2$ above $H_1$. Let $V_1=R(H_1)$ and $V_2=R(H_2)$ be two vertical strips. For $i,j \in \{1,2\}$, denote by $V_{ji}=R(H_i)\cap H_j$ and $H_{ij}=R^{-1}(V_{ji})$.  A continuous map $R$ defined on $H_1\cup H_2$ is called a \emph{regular horseshoe} if and only if
\medbreak
\begin{enumerate}
\item the map $R$ maps $H_{ij}$ homeomorphically into $V_{ji}$;
\item vertical boundaries of $H_i$ are mapped onto vertical boundaries of $V_i$. The vertical boundaries of $V_{ji}$ are contained in the vertical boundaries of $V_i$.
\item if $H$ is a horizontal strip which intersects $H_j$ fully, then $R^{-1}(H)\cap H_i$ is a horizontal strip intersecting $H_i$ fully, for $i\in\{1,2\}$ and $\diam( R^{-1}(H)\cap H_i)<\nu \diam (H)$ for some $\nu \in (0,1)$; the analogous should hold for vertical strips. 
\end{enumerate}
\end{definition} 
 
\begin{definition}
\label{cusp_horseshoe_def}
Let $H_1$ and $H_2$ be two disjoint horizontal strips across $\Sigma_P^{\inn}$ with $H_2$ sitting above $H_1$. Let $V_1=R(H_1)$ and $V_2=R(H_2)$ be two vertical cusps which have only the cusp point $p^\star$ in common. A continuous map $R$ defined on $H_1\cup H_2$ is called a \emph{cusp horseshoe} if and only if
\medbreak
\begin{enumerate}
\item vertical boundaries of $H_i$ are mapped onto vertical boundaries of $V_i$. One of the horizontal boundaries of $H_i$, say $\ell_i$, is mapped into the cusp point and the other, denoted by $u_i$ is mapped onto the horizontal boundary of $V_i$. Moreover, $R$ maps $H_i \backslash \{\ell_i\}$ homeomorphically into $V_i \backslash \{p^\star\}$;
\medbreak
\item there is a constant $\mu \in(0,1)$ such that if $H$ is a horizontal strip in $H_1 \cup H_2$ and the cusp point lies below or on the lower horizontal boundary of $H$, then $R^{-1}(H) \cap H_i$ is a horizontal strip in $H_i$ and $$\diam (R^{-1}(H) \cap H_i)\leq \nu \diam (H), \qquad \text{ for some} \qquad \nu \in (0,1).$$ Similarly if $V$ is a vertical cusp in $V_1 \cup V_2$ with the same cusp point $p^\star$ and it lies below or on the lower boundary of $H_i$ then $R(V) \cap V_i$ is a vertical cusp in $V_i$ with the same cusp point and $\diam(R(V) \cap V_i)\leq \nu \diam(V)$.
\end{enumerate}
\end{definition}

\subsection{The organizing center} 
In the spirit of Homburg \emph{et al} \cite{HKK}, we are now in a position to state our main theorem about the previously defined horseshoes in the organizing center. It is convenient to introduce the notation $\delta:=\delta_1 \delta_P$.
\begin{theorem}\label{main-thm}
The properties listed in Table \ref{dynamics} are satisfied by an open and dense set of $C^\infty$-vector fields on $M$ that satisfy $(\mathbf{H1})-(\mathbf{H4})$, depending on the signs of the nonzero parameters $a_{21}$, $k$ and $b_{22}$:
\begin{table}[ht]
\begin{center}
\begin{tabular}{cllllcc}
Case & $a_{21}$ & $k$ & $b_{22}$ & Dynamics &  Fig.~\ref{horseshoes} &  Fig.~\ref{multipulses}\\ \hline
1 & negative & any & any & Isolated cycle & -& -\\\hline 
2 & positive  & negative & negative &   Isolated cycle & - & I\\ \hline
3& positive    & positive & negative & Cusp horseshoe if $\delta_P \delta_2>1$ & (a) & II\\ \hline
4& positive    & negative & positive & Cusp horseshoe if $\delta<1$ & (b)& III\\ \hline
5& positive    & positive & positive & Regular horseshoe if $\delta>1$ & (c) & IV\\ \hline
\hline
\end{tabular}
\end{center}
\caption{Dynamics near $\Gamma$ depending on the sign of the parameters and the schemes of Figures ~\ref{horseshoes} and  \ref{multipulses}.} 
\label{dynamics}
\end{table}
\end{theorem}

Hereafter we will use the classification given by Table \ref{dynamics}. For instance, when we say \emph{Case 5}, we refer to the case where $a_{21}>0$, $k>0$ and $b_{22}>0$. 

\begin{proof}We proceed case by case, according to Table \ref{dynamics}.  The constants are nonzero. Recall the roles of $a_{21},k,b_{22}$ from Table \ref{notation}. In Cases 1 and 2, the recurrent dynamics consist of the cycle, which is Lyapunov unstable, as we shall observe. We explain step by step all the details for Case 3. For Cases 4 and 5, in order to facilitate the reading, we just sketch the proof; the arguments are similar, with appropriate adaptations. 
\begin{itemize}
\medbreak
\item[\textbf{Case 1:}] This case was already discussed in Remark \ref{remark-a21}. The dynamics in a neighbourhood $\mathcal{U}$ of the cycle are trivial: solutions that start near the cycle are thrown away from it in the transition from $\Sigma_E^{\out}$ to $\Sigma_P^{\inn}$. 
\medbreak
\item[\textbf{Case 2:}]  Where $a_{21}>0$ and $k,b_{22}<0$, the cycle is also isolated, because trajectories stop following it along the connection $[P \to E]_t$. More specifically, trajectories starting at $\mathcal{U}\cap \Sigma_P^{\inn(+)}$ are mapped into $\Sigma_P^{\out(+)}$ (where $\rho>0$) and then follow the lower part of $\Psi_{PE}(W^u_{\loc}(P)\cap \Sigma_P^{\out})\subset \Sigma_E^{\inn(-)}$. Therefore, they leave any small neighbourhood of the cycle.
\medbreak
\item[\textbf{Case 3:}] We have $a_{21},k>0$ and $b_{22}<0$. We would like to apply \cite[Theorem 4.7]{Deng} to conclude the existence of a cusp horseshoe and thus shift dynamics. By the remark after Lemma \ref{main_prop}, there is a sequence of strips (see (\ref{Hn_def}))
$${H}_n = [-\tau, \tau] \times  [\exp({a}_n(\psi)), \exp({b}_n(\psi))] \subset R_P^{\inn},  \quad n \geq n_0 \in \mathbb{N}, \quad \psi \in[-\tau, \tau],$$
such that $\widetilde{H}_n=\Pi_P(H_n)$ is a sequence of horizontal strips in $R_P^\out \subset \Sigma_P^{\out}$, accumulating on $W^u_\loc(P)$. For each such strip, one of two situations occurs: either $W^s_{\loc}(E)$ does not intersect $\widetilde{H}_n$,  meaning that all solutions fall directly into $\Sigma_E^{\inn(-)}$, or  $W^s_{\loc}(E)$ intersects transversely the boundaries of the strip, implying that part of the strip is mapped into $\Sigma_E^{\inn(+)}$.  For fixed $\tau>0$ and $n\in\NN$ large enough, say $n>n_1>n_0$, the latter always happens. 
\medbreak
The parts of $H_n$ that are mapped under $\Psi_{PE}\circ \Pi_P$ into $\Sigma_E^{\inn(+)}$ define two disjoint horizontal strips that we denote by $H_n^1$ and $H_n^2$. This pair of strips is well defined because the set $\Pi_P^{-1}(W_\loc^s(E)\cap \Sigma_P^\out)$ is a double helix in $\Sigma_P^{\inn}$ dividing the strip $H_n$ in two horizontal strips. Initial conditions in $H_n^1$ and $H_n^2$ follow (at least once) the cycle. By Proposition \ref{dyn_E_proposition} (3) these are mapped into cuspidal regions across $\Sigma_E^{\out}$ and thus, by $\Psi_{EP}$, into cuspidal regions across $\Sigma_P^{\inn}$.
\medbreak
Let $h_{n,j}=h_{n,j}^{\Max}$ be the height of $H_n^j$. Hereafter, we omit the letter $j$ in $h_{n,j}$ to simplify the notation. If $n$ is large enough  then, for $j \in \{1,2\}$, the width of $\widetilde{H}_n^j$ is of order $\tau - \left(\frac{h_n}{k} \right)^{\frac{\delta_P}{2}}$.
So in order to intersect the original $H_n^j$, the height of $\widetilde{H}_n^j$ under $\Psi_{EP}\circ \Pi_E \circ \Psi_{PE}$ should satisfy the inequality 
\begin{equation}\label{case3}
\tau - \left(\frac{h_n}{k} \right)^{\frac{\delta_P}{2}}>h_n^\frac{1}{\delta_1}, 
\end{equation}
which is always true for $h_n$ very small (meaning $n\in \NN$ large enough), see Figure \ref{graph1}.

\begin{figure}[ht]
\begin{center}
\includegraphics[height=5.5cm]{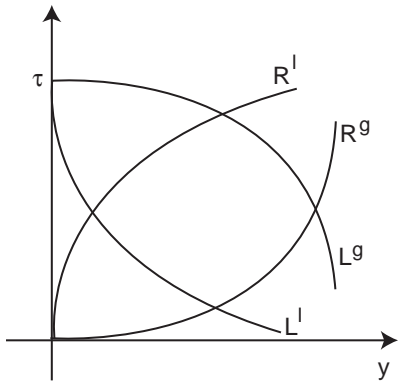}
\end{center}
\caption{\small Graph of the two maps of equation (\ref{case3}) in Case 3: $L(y)=\tau - \left(\frac{y}{k} \right)^{\frac{\delta_P}{2}}$ and $R(y)=y^\frac{1}{\delta_1}$. The exponents $l$ and $g$ mean that the constants $\frac{1}{\delta_1}$ and $\frac{\delta_P}{2}$ are less or greater than 1, respectively.}
\label{graph1}
\end{figure}
\medbreak

Employing the notation $V_n^j:=R(H_n^j)$ for $j \in \{1,2\}$, we check the assumptions of Definition \ref{cusp_horseshoe_def}:
\medbreak
\begin{itemize}
 \item By Remark \ref{boundaries}, vertical (horizontal) boundaries of $H_n^j$ are mapped onto horizontal (vertical) boundaries of $\widetilde{H}_n^j$. From this, we conclude that the vertical boundaries of $H_n^j$ are mapped by $R$ onto the vertical boundaries of $V_n^j$ for $j \in \{1,2\}$. By construction, one of the vertical boundaries of each $\widetilde{H}_n^j$ lies in $W^s(E)$;
 \medbreak
  \item one of the horizontal boundaries of $H_n^j$, say $\ell_n^j$, is mapped (by continuity) into the cusp point $p^\star=(0,0)\in \Sigma_P^{\inn}$, the other is mapped onto the horizontal boundary of $V_n^j$;
  \medbreak
 \item by construction, $R$ maps each $H_n^j\backslash \ell_n^j$ homeomorphically onto $V_n^j \setminus \{p^\star\}$;
 \medbreak
 \item  if $H$ is a horizontal strip in $H_n^1\cup H_n^2$ then it is easy to check that $\hat{H}=R^{-1}(H)\cap H_n^i$, $i=1,2$, is again a horizontal strip in $H_n^1\cup H_n^2$. Moroever, if the height of $H$ is small enough there exists $k\in \RR^+$ such that $ \diam (\hat{H}) < k \diam (H)^{\delta_P \delta_2}< k \diam(H)$. The last equality holds because $\delta_P \delta_2>1$.  Similarly for vertical strips.
 \end{itemize}
 \medbreak
Therefore, we can apply \cite[Theorem 4.7]{Deng} to conclude that there exists a cusp horseshoe and thus a subset $A \subset \Sigma_P^{\inn}$ such that the first return map to it is topologically conjugate to a full shift. The effect of the return map on the rectangles is illustrated in Figure \ref{horseshoes} (a). 
\medbreak

\item[\textbf{Case 4:}] 
Here we have $a_{21}, b_{22}>0$ and $k<0$. We use the same method from \cite{Deng} as in Case 3 to conclude the existence of chaos. 
We proceed using the sequence of rectangles ${H}_n$ as in Case 3. Since $\Psi^{-1}_{PE}(W^s_{\loc}(E)\cap \Sigma_E^{\inn})$ is a parabola intersecting $\widetilde{H}_n$ for $n$ large enough, the set $\Psi_{PE}(\widetilde{H}_n)$ is a horseshoe strip in $\Sigma_E^{\inn(+)}$. Then by Proposition \ref{dyn_E_proposition}, $R(H_n)$ is bounded by closed cusp curves tangent to $E^{cu}$ and centered at the heteroclinic connection that meets $\Sigma_P^{\inn}$. Whether or not this image intersects the original $H_n$ as desired depends on the height of the bounding cusp curves, \emph{i.e.}\ their maximal distance from $W^s_\loc(P)$ in $\Sigma_P^\inn$, in comparison to the height of $H_n$. The latter is of order $\exp(a_n)$ and under $R$ evolves to the order $\exp(a_n)^\delta$. Since $\exp(a_n) \to 0$ for $n \to \infty$, we have
$$
\exp(a_n)^\delta>\exp(a_n) \Leftrightarrow \delta<1
$$
for large $n$, say $n>n_1$, and so the strips $H_n$ and $R(H_n)$ intersect if $\delta<1$. 
\medbreak
By possibly making $H_n$ thinner we can assure just as in the previous case, that assumptions (1) and (2) of Definition \ref{cusp_horseshoe_def}  are satisfied, with the exception that both horizontal boundaries are mapped into the cusp point $p_\star \in \Sigma_P^\inn$ instead of just one. Therefore, the argument of \cite[Theorem 4.7]{Deng} may be applied along the same lines. The effect of the return map on the strips is illustrated in Figure \ref{horseshoes} (b).  

\medbreak

\item[\textbf{Case 5:}]
Here we have $a_{21}, b_{22}, k>0$. We would like to apply  \cite[Theorem 2.3.3]{Wiggins} of Wiggins. Still the method of proof is similar to the previous cases. Again we look at the rectangles $H_n \subset R_P^\inn$ and their images $\widetilde{H}_n=\Pi_P(H_n) \subset R_P^\out$. In this case, no rectangle $\widetilde{H}_n$, $n>n_0$, intersects the stable manifold of $E$, which is locally contained in the lower part of $\Sigma_P^\out$. The images $\Psi_{PE}(\widetilde{H}_n)$ thus form a sequence of horizontal strips accumulating on the parabola $\Psi_{PE}(W^u_{\loc}(P)\cap \Sigma_P^{\out})$. Under $\Psi_{EP}\circ \Pi_E$ this becomes a sequence of strips in $\Sigma_P^\inn$, bounded by curves with a fold point and accumulating on the cusp curve given through $W^u(P) \subset \Sigma_P^\inn$. Now we get the desired intersection of $R(H_n)$ and $H_n$ if the distance of the fold point from $W^s_\loc(P)$ is less than the height of $H_n$. The latter evolves as in Case 4, and so the strips $H_n$ and $R(H_n)$ intersect if $\delta>1$. Assumptions (1)--(3) in Definition \ref{reg_horseshoe_def} can be checked similarly to the previous cases. The effect of the return map on the strips is illustrated in Figure \ref{horseshoes} (c).
\end{itemize}
\end{proof}

\begin{figure}
\begin{center}
\includegraphics[width=14cm]{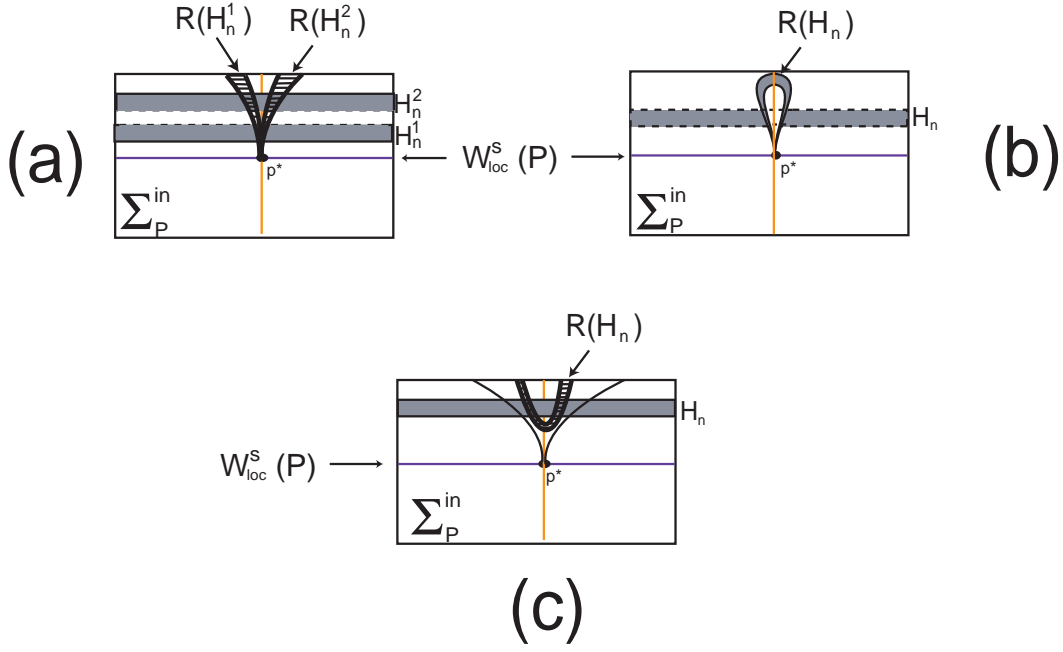}
\end{center}
\caption{\small Shape of the horseshoes that appear in the organizing center. (a) Case 3 of Theorem \ref{main-thm}: $a_{21}>$, $k>0$ and $b_{22}<0$. (b) Case 4 of Theorem \ref{main-thm}: $a_{21}>0$, $k<0$ and $b_{22}<0$.  (c) Case 5 of Theorem \ref{main-thm}: $a_{21}>$, $k>0$ and $b_{22}>0$. Cases (a) and (b) are examples of cusp horseshoes. Case (c) is an example of a regular horseshoe.}
\label{horseshoes}
\end{figure}

In all cases where we observe chaos in the organizing center, the non-trivial dynamics of $R$ is concentrated in the rectangle $R_P^{\inn}$. Let 
\begin{equation}
\label{saturated}
\Omega_{R} = \bigcap_{n \in \ZZ} \overline{R^n(R_P^{\inn})}.
\end{equation} The non-wandering set associated to $\Gamma$ contains the disjoint union of the critical elements $\{E,P, [E \rightarrow P], [P \rightarrow E]_t\}$ and $\Lambda$, where:
\begin{equation}
\label{Lambda}
\Lambda= \Lambda^\infty \cup \Lambda^s \cup \Lambda^u \cup \Lambda^h
\end{equation}
 is the saturation (by the flow) of $\Omega_{R}$, with:
\begin{itemize}
\item $\Lambda^\infty $ is the saturated regular horseshoe, sometimes the so called \emph{basic set};
\item $\Lambda^s$ is the set of points that escape in negative time to the cusp point;
\item $\Lambda^u$ is the set of points that escape in positive time to the cusp point;
\item $\Lambda^h$ keeps track of the trajectory in the unstable manifold of $E$. 
\end{itemize}
In \cite{LP}, the invariant set $\Lambda$ is called a singular horseshoe.
By construction (see (\ref{saturated})), the set $\Lambda$ is the maximal invariant Cantor set contained in $\mathcal{U}$, the neighbourhood of $\Gamma$ chosen at the beginning of the construction. 


\subsection{Hyperbolicity}
From the beginning of Section \ref{organizing_center}, recall that $R = \Psi_{EP}\circ \Pi_E \circ \Psi_{PE} \circ \Pi_P$ is the first return map on the union of strips  $\bigcup_{n\geq n_0}H_n$ (or $\bigcup_{n\geq n_0} H_n^j$ for Cases 3 and 4, where $j \in \{1,2\}$) contained in $R_P^{\inn}$ and centered at the heteroclinic connection. In this section we show that $R$ with the expression given by
\begin{equation*}
\label{first return}
R(\psi, z)=\left[-\left(b_{11}A_n+b_{12}z^{\delta_p}\right)(b_{22}z^{\delta_p}+kA_n^2)^{\delta_2}, (b_{22}z^{\delta_p}+kA_n^2)^{\delta_1}\right] 
\end{equation*}
with 
$$
A_n(\psi, z)=\psi-\frac{1}{\eta_e} \ln z-2n\pi\geq 0,
$$
is uniformly hyperbolic restricted to a compact set disjoint from the line $z=0$ and containing $\Lambda$, which is important to establish the persistence of dynamics.

\begin{theorem}\label{thm-hyp}
Under the conditions of Theorem \ref{main-thm}, there exists a sequence of compact sets within the cross section $R_P^{\inn}$ where the first return map to $R_P^{\inn}$ is uniformly hyperbolic and conjugated to a full shift over a finite number of symbols. This sequence accumulates on the cycle.
\end{theorem}

\begin{proof}
The existence of the subset $\Omega_R \subset \Sigma_P^{\inn}$ such that the map $R$ is topologically conjugate to a Bernoulli shift follows from Theorem \ref{main-thm}. The sequence in the statement is the one defined in Lemma \ref{main_prop}.  In this computer assisted proof, we assume without loss of generality that $\eta_e=2$, $b_{11}= b_{12}=1$ and that the transition $\Psi_{EP}$ is a rotation by $\pi/2$ (see Section \ref{local}). 

\medbreak
The map $R$ is hyperbolic at the point $(\psi,z)$ if both $R$ and $R^{-1}$ are well defined in the Cantor set $\Omega_R$ -- see (\ref{Lambda}) -- and if there is $\sigma \in (0,1)$ such that in suitable coordinates the maps $\mathrm{D}R(\psi,z)$ and $\mathrm{D}R^{-1}(\psi,z)$ satisfy:
\begin{enumerate}
\item the sector bundle $S^u_\sigma=\{(\psi,z): |\psi|<\sigma|z|\}$ is invariant under $\mathrm{D}R$ and $\mathrm{D}R(S^u_\sigma)\subset S^u_\sigma$;
\item  the sector bundle $S^s_\sigma=\{(\psi,z): |z|<\sigma|\psi|\}$ is invariant under $\mathrm{D}R^{-1}$ and $\mathrm{D}R^{-1}(S^s_\sigma)\subset S^s_\sigma$;
\item $\forall (\psi,z) \in S^u_\sigma$ if $\mathrm{D}R(\psi,z)=(\tilde{\psi},\tilde{z})$ then $|\tilde{z}|\geq \sigma^{-1} |\tilde{z}|$ and
\item $\forall (\psi,z) \in S^s_\sigma$ if $\mathrm{D}R^{-1}(\psi,z)=(\tilde{\psi},\tilde{z})$ then $|\tilde{\psi}|\geq \sigma^{-1} |\tilde{\psi}|$.
\end{enumerate}
Using the software Maple, up to higher order terms, the determinant of the Jabobian matrix $\mathrm{D}R$ evaluated at $(\psi, z)$ can be written as
$$
\det \mathrm{D}R(\psi,z)= k_1 z^{\delta_P -1} g_1^{\delta_1+\delta_2-1} \quad \text{where}\quad k_1 \in \RR^+ \quad \text{and}\quad g_1(\psi,z)= z^{\delta_P}+k\psi^2 -k\psi \ln z + k(\ln z)^2.
$$
Similarly, when $z$ goes to $0$, we obtain for the trace
$$
\tr \ \mathrm{D}R (\psi,z) \approx \left(-\frac{1}{\eta_e z}+\delta z^{\delta_P-1} \right) \left(z^{\delta_P} + A_n^2\right)^{\delta_P}+ \frac{\delta_2}{z} A_n^{2\delta_2+1} -2\delta_1 A_n^{2\delta_1-1}.
$$
It is easy to see that $$\lim_{z \rightarrow 0^+} \det \mathrm{D}R(\psi,z)=0$$ and $$\lim_{z \rightarrow 0^+} |\tr \ \mathrm{D}R(\psi,z)|=\lim_{z \rightarrow 0^+} |\ln z|>\lambda^\star \quad \text{for some} \quad \lambda^\star>2.$$ 

This property allows one to apply the construction of appropriate families of cones as described in \cite{KH95} and so conclude that $\Lambda$ is a hyperbolic invariant set for the first return map $R$. Indeed, the set
$S^u_\sigma =\{(\psi,z)\in \mathbb{R}^2: |\psi|\leq \sigma |z|\} \supset \left\langle (0,1)\right\rangle$
is an unstable cone-field for $\frac{1}{\lambda^\star-1}<\sigma<1$; such a choice is possible because $\lambda^\star>2$. Indeed, if we 
write $\mathrm{D}R(\psi,z)$ as $(\psi',z')$, we see that, for $z$ sufficiently small, we get
$$
|\psi'|=|z|\leq \frac{1}{\lambda^\star-\sigma} \left|\frac{\partial R}{\partial \psi} z-\frac{\partial R}{\partial z} \psi\right| \approx \frac{1}{\lambda^\star-\sigma}|z'| 
$$
so that $\mathrm{D}R(S^u_\sigma)\subset S^u_{\theta \sigma}$ where $\theta=(\sigma(\lambda^\star-\sigma))^{-1}<1$ by the choice of the parameter $\sigma$. That is, $S_\sigma^u$ is $\mathrm{D}R$-invariant. Furthermore, denoting by $\|(\psi,z)\|=\max\{|\psi|,|z|\}$, we get, for any $(\psi,z)\in S^u_\sigma$,
$$
\|\mathrm{D}R(\psi,z)\|=|z'|\geq (\lambda^\star-\sigma)|z| =(\lambda^\star-\sigma)\|(\sigma,z)\|
$$
with $(\lambda^\star-\sigma)>1$, \emph{i.e.}, $\mathrm{D}R$ uniformly expands any vector inside $S^u_\sigma$. On the other hand, it is not hard to see that the
same above argument can be applied to $\mathrm{D}R^{-1}$ in order to find a stable cone-field ($S^s_\sigma  \supset \left\langle (1,0)\right\rangle $). Using the invariant cone-field criterion by Katok and Hasselblatt \cite[Corollary 6.4.8]{KH95}, the proof ends.  Every compact invariant subset of the cross section $\Sigma^{\inn}_P$ containing $\Omega_R$ and not containing the line $z=0$ is uniformly hyperbolic (in particular, this part persists under small generic perturbations).
\end{proof}

\begin{remark}
Using the sequences $\exp(a_n)_{n\geq n_0}, \exp(b_n)_{n\geq n_0}$ derived in Lemma \ref{main_prop},  we may define a sequence of invariant hyperbolic sets on which the related first return map to the set parametrized by $[-\tau, \tau] \times [\exp(a_n), \tau]$  is conjugated to the symbolic system $(Y_n, \sigma)$, where the number of symbols increases with $n$. It means that the topological entropy is increasing with $n$.
\end{remark}

\begin{remark}
Notice that the invariant set $\Lambda$ (see (\ref{saturated}) and (\ref{Lambda})) is not compact since it accumulates on $W^s(P)\cap W^u(E)$, where the first return map $R$ is not defined. If $\Lambda$ were uniformly hyperbolic with respect to the flow, this hyperbolicity would uniformly spread to the solutions in the closure of $\Lambda$, which is impossible due to the presence of the equilibrium $E$.
\end{remark}

\begin{remark}
As the diffeomorphism $R$ is (at least) $C^2$ and $\Lambda$ is a countable union of uniformly hyperbolic horseshoes, the two-dimensional Lebesgue measure of $\Lambda$ is zero. Additionally, the union of $W^s(\Lambda)$, $W^u(\Lambda)$, $W^s(E)$, $W^s(P)$, $W^u(E)$ and $W^u(P)$ has zero three-dimensional Lebesgue measure.
\end{remark}

\begin{remark}
The horseshoes encoded in $\Lambda$ are related to the heteroclinic tangles. The information given by the proof of Theorems \ref{main-thm} and \ref{thm-hyp}  suggests that $W^s(P) = W^s(\Lambda)$ and $W^{cu}(E) = W^u(\Lambda)$. The role of these manifolds in the overall topological structure of the horseshoe remains future work.
\end{remark}

\section{Global bifurcations: multipulses}\label{sec-multipulses}
In this section we study the creation and disappearance of a suspended horseshoe through sequences of homo- and heteroclinic bifurcations. In particular, we investigate the existence of multipulse $E$- and $P$-homoclinic trajectories, \emph{i.e.}\ solutions bi-asymptotic to $E$ or $P$ that pass more than once around the heteroclinic cycle (or its remnants). In \cite[\textbf{(Q1)}]{Knobloch_open} these multipulses are described as ``\emph{two or more copies of the localized states}''. 

From now on we concentrate mainly on the situations in which there is chaos in the organizing center, \emph{i.e.}\ Cases 3--5 in Theorem \ref{main-thm}. If a heteroclinic cycle exists for $(\alpha, \beta)$ in the parameter space, we denote it by $\Gamma_{(\alpha, \beta)}$. For $(\alpha, \beta)$ close to $(0,0)$ we fix a sufficiently small neighbourhood $\mathcal{U} \subset M$ of the original cycle $\Gamma_{(0,0)}$, that contains the neighbourhoods $V_E$ of $E$ and $V_P$ of $P$ from before. This allows us to define multipulse homoclinic cycles also for parameters $(\alpha, \beta)$ for which there is no heteroclinic cycle: we say that a trajectory starting in $V_E$ ($V_P$) \emph{takes a turn around the (remnant of the) heteroclinic cycle}, if it leaves $V_E$ ($V_P$), then passes through $V_P$ ($V_E$) and comes back to $V_E$ ($V_P$) without leaving $\mathcal{U}$ in between.

\subsection{Counting sections and notation}
In order to give a precise meaning to the phrase \emph{pass around the periodic solution $P$}, we recall the following concept of counting sections for periodic solutions from Rodrigues \emph{et al.}\ \cite{RodLabAgu_2011}. See also the concept of \emph{winding number} of Rademacher \cite{Rad_PhD}. In a neighbourhood $V_P$ of $P$ as before, consider a codimension one submanifold  $C \subset V_P\subset M$ with boundary, such that
\begin{itemize}
\item the flow is transverse to $C$;
\item $C$ intersects $\partial V_P$ transversely;
\item $W^s_{\loc}(P) \subset \partial C$ and $W^u_{\loc}(C) \subset \partial C$.
\end{itemize}
We call $C$ a \emph{counting section}. We are interested in trajectories that enter the neighbourhood $V_P$ in positive time and hit the counting section $C$ a finite number of times (which can be zero) before they leave $V_P$ again. Every time the trajectory makes a turn around $P$ inside $V_P$, it hits the counting section. This behaviour is illustrated in Figure \ref{counting-section}. It is natural to have the following definition, where $\Int(A)$ is the topological interior of $A \subset M$:

\begin{definition}
Let $V_P$ be {an isolating block for} $P$, and $C$ a counting section as defined above. Let $q \in \partial V_P$ be a point such that
\begin{equation*}
\exists \tau>0: \ \forall t \in (0,\tau): \ \varphi(t,q) \in \Int(V_P), \quad \text{and} \quad  \varphi(\tau,q) \in \partial V_P.
\end{equation*}
For $n \in \NN$ the trajectory  of $q$ \emph{turns $n$ times around $P$  in $V_P$, relative to $C$} if
\begin{equation*}
\#\left(\{\varphi(t,q) : \ t\in[0,\tau]\}\cap C \right)= n\geq 0 \ .
\end{equation*}
\end{definition}
Considering our previous notation we use $\CC=\left\{ (0,\rho,z) \in V_P : \ 0\leq \rho,z \leq 1 \right\} \subset \mathcal{U}$ as our counting section. With respect to $\CC$ we now define multipulse homoclinic cycles and heteroclinic connections.

\begin{figure}[ht]
\begin{center}
\includegraphics[width=7cm]{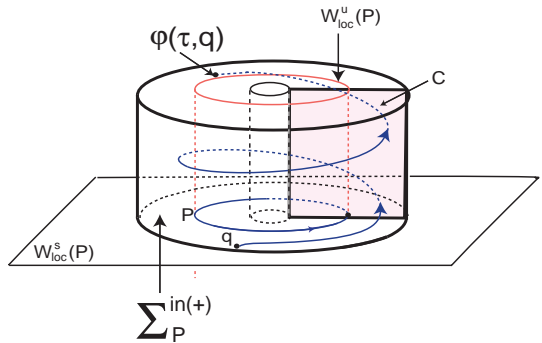}
\end{center}
\caption{\small Example of a trajectory turning twice around the periodic solution $P$ with respect to the cross section $C$.}
\label{counting-section}
\end{figure}

\begin{definition}\label{def-multipulses}
Let $k, n_1, n_2, ..., n_k \in \NN$.
Denote by
\begin{enumerate}
\item $\Hom_{n_1, \ldots, n_k}(E)$ the set of points $(\alpha, \beta)$ in the parameter space such that the flow of (\ref{general}) has a homoclinic cycle associated to $E$, taking $k$ turns around the (remnant of the) heteroclinic cycle, turning $n_j$ times around $P$ in the $j$-th round, for $j\in \{1,\ldots,k \}$.
\medbreak
\item $\Het_{n_1, \ldots, n_k}(PE)$ the set of points $(\alpha, \beta)$ in the parameter space such that the flow of (\ref{general}) has a tangent heteroclinic connection from $P$ to $E$, taking $k$ turns around the (remnant of the) heteroclinic cycle, turning $n_j$ times around $P$ in the $j$-th round, for $j \in \{1,\ldots,k \}$.
\medbreak
\item $\Hom_{n_1,\ldots,n_{k-1}}(P)$ the set of points $(\alpha, \beta)$ in the parameter space such that the flow of (\ref{general}) has a $P$-homoclinic tangency taking $k$ turns around the (remnant of the) heteroclinic cycle. For $t$ bounded appropriately, such a trajectory visits $V_P$ precisely $k-1$ times, turning $n_j$ times around $P$ in the $j$-th visit, for $j \in\{1, \ldots ,k-1 \}$. For a $P$-homoclinic solution taking just one turn around the (remnant of the) cycle, we simply write $\Hom(P)$.
\end{enumerate}
\end{definition}

For $a_{21}>0$, the two curves
$$
\left\{
\begin{array}{l}
\beta=-b_{22}\alpha^{\delta_P} \\
\alpha\geq 0
\end{array}
\right.
\qquad
\text{and}
\qquad
\left\{
\begin{array}{l}
\beta=\tilde{k} \alpha^{\frac{1}{\delta_1}}\\
\tilde{k} \in \RR^+, \alpha\leq 0
\end{array}
\right.
$$
divide the bifurcation diagram $(\alpha, \beta)$ in two connected components; one will be called \emph{chaotic region} and the other \emph{regular region}. The chaotic region corresponds to that for which we observe chaos for values in the half $\beta$-axis that is contained on it. A geometrical interpretation of these curves will be given in Theorems \ref{thm-E-hom} and \ref{thm-P-hom}. These regions are depicted in Figure \ref{multipulses}. 

\subsection{Homoclinic cycles to $E$}
In the next result, we show the existence of infinitely many curves in the parameter space for which we observe homoclinic trajectories of $E$ taking more and more turns around the original cycle (or its remnants). The proof of item (1) is based on \cite[Section 3.2]{Champneys2009}. The existence of periodic solutions (but not their stability) as in statement (5) has also been derived in \cite[Section 6.1]{Champneys2009}. For alternative proofs of items (1) and (2) using Lin's method see also \cite{KnobRiez2010, Knob_JDDE_2011}. However, this method is unable to make stability statements.

\begin{theorem}\label{thm-E-hom}
Consider a vector field $f_0$ satisfying $\textbf{(H1)}-\textbf{(H5)}$, with $a_{21}>0$ (Cases 2--5). Then, the dynamics of a generic two-parameter family of $C^\infty$ vector fields $f (x,\alpha,\beta)$ unfolding $f_0$, for $\alpha,\beta\neq 0$ sufficiently small, satisfy:
\begin{enumerate}
\item there are infinitely many curves $\Hom_{n}(E)$ where $n \geq n_1 \in \NN$, whose tip points lie on $\beta=-b_{22}\alpha^{\delta_P}$ with $\alpha>0$.
\item the curves $\Hom_{n}(E)$ accumulate on the non-positive $\beta$-axis if $k>0$ and on the non-negative $\beta$-axis if $k<0$.
\item in Cases 2 and 5, for each $n \geq n_1 \in \NN$, there are two curves $\Hom_{n, n}(E)$ on the left side of the curve $\Hom_{n}(E)$ whose tip points coincide with the tip point of $\Hom_{n}(E)$. Moreover, for $m>n$ each curve $\Hom_{m}(E)$ bifurcates into two curves $\Hom_{n,m}(E)$ at each intersection point between $\Hom_{m}(E)$ and $\beta=\beta_\star$, where $(\alpha_\star, \beta_\star)$ are the coordinates of the fold point of $\Hom_{n}(E)$. 
\item in Cases 3 and 4, for $m \geq n \geq n_1$ there are curves $\Hom_{n,m}(E)$ bifurcating from $\Hom_{m}(E)$ at its tip point.
\item bifurcating from each of the curves in (1), (2) and (4) there are periodic saddles or sources depending on the condition $\lambda_1+\lambda_2 \gtrless \mu$. These periodic solutions occur on the left side of the curves, i.e.\ they are created for decreasing values of $\alpha$. Their periods diverge to $+\infty$ when they approach the curve they bifurcate from.
\end{enumerate}
\end{theorem}

\begin{proof}
\begin{enumerate}
\item
For $(x,y) \in \Sigma_E^{\out}$ we have
$$\Pi_P \circ \Psi_{EP}(x,y)=\left(a_{11}x+a_{12}y+ \alpha \xi_1 -\frac{1}{\eta_e}\ln(a_{21}x+a_{22}y+\alpha), \ (a_{21}x+a_{22}y+\alpha)^{\delta_P}\right) \in \Sigma_P^{\out},$$
which is mapped under $\Psi_{PE}$ to a point in $\Sigma_E^{\inn}$ whose $z$-coordinate is given by
$$
b_{22} \rho+ k \psi^2 + \beta(1+\nu \psi)=0, \qquad \text{where} \qquad 
\left\{ 
\begin{array}{l}
\psi=a_{11}x+a_{12}y+ \alpha \xi_1 -\frac{1}{\eta_e}\ln(a_{21}x+a_{22}y+\alpha) \\ 
\rho=(a_{21}x+a_{22}y+\alpha)^{\delta_P}
\end{array}.
\right.
$$

Since we are looking for homoclinic cycles to $E$ we have to check when the origin $(0,0)\in \Sigma_E^{\out}$ falls on the line $z=0$ in $\Sigma_E^\inn$. Note that this can only happen if $\alpha > 0$, since otherwise $W^s_\loc(E) \cap \Sigma_E^\out$ is mapped into the lower part of $\Sigma_P^\inn$ and thus leaves $\mathcal{U}$. Therefore, with $x=y=0$, we get:
\begin{equation}
\label{eq0}
b_{22}\alpha^{\delta_P} + k \left(\alpha \xi_1 -\frac{1}{\eta_e}\ln \alpha \right)^2 + \beta \left(1+\nu \left(\alpha \xi_1 -\frac{1}{\eta_e}\ln \alpha \right) \right) =0.
\end{equation}

Let $A_n = \alpha \xi_1 -\frac{1}{\eta_e}\ln \alpha-2n\pi\geq0$. Then, condition (\ref{eq0}) is equivalent to the  following equation of order 2 in the variable $A_n$:
\begin{equation}
\label{eq1}
 k A_n^2  + \nu \beta A_n + \left(b_{22} \alpha^{\delta_P}+\beta \right)=0.
\end{equation}
Solutions of (\ref{eq1}) to leading order can be written as:
 \begin{eqnarray*}
 A_n=\frac{-\nu \beta\pm \sqrt{(\nu\beta)^2-4k(b_{22}\alpha^{\delta_P}+\beta)}}{2k}
& \Leftrightarrow & A_n=- \frac{\nu\beta}{2k} \pm \sqrt{\frac{(\nu \beta)^2}{4k^2}-\frac{1}{k}{\left(b_{22}\alpha^{\delta_P}+\beta\right)}}\\
& \Leftrightarrow & A_n= - \frac{\nu\beta}{2k} \pm  \sqrt{\frac{1}{k}{\left(-b_{22}\alpha^{\delta_P}-\beta\right)}}
\end{eqnarray*}

Recalling the definition of $A_n$, we may write:
$$
\alpha \xi -\frac{1}{\eta_e} \ln\alpha -2n\pi = -\frac{\nu\beta}{2k} \pm  \sqrt{\frac{1}{k}{(-b_{22}\alpha^{\delta_P}-\beta)}}.
$$
Since $\alpha \xi_1 -\frac{1}{\eta_e} \ln\alpha\approx -\frac{1}{\eta_e} \ln\alpha$ when $\alpha \approx 0$, and thus
\begin{equation}\label{eq2}
-\frac{1}{\eta_e}\ln\alpha \approx 2n\pi -\frac{\nu \beta}{2k}\pm \sqrt{\frac{1}{k}{(-b_{22}\alpha^{\delta_P}-\beta)}},
\end{equation}
where $n\geq n_1$ is the number of times the homoclinic cycle winds around $P$ according to Definition \ref{def-multipulses} and $n_1$ is the smallest natural number such that $A_{n_1}>0$. Then homoclinic cycles to $E$ are created along the curve that is to first order given by $\beta=-b_{22}\alpha^{\delta_P}$. The fold point where the curve $\Hom_n(E)$ is created is called the tip point of the curve.
\medbreak

\item From the expression under the square root in equation (\ref{eq2}) it is clear that the sign of $k$ determines on which side of the curve $\beta=-b_{22}\alpha^{\delta_P}$ the $E$-homoclinics appear: if $k>0$, then there are double roots for $\beta<-b_{22}\alpha^{\delta_P}$, so they occur in the lower side. If $k<0$, double roots exist for $\beta>-b_{22}\alpha^{\delta_P}$ and the homoclinics occur in the upper side. Thus, the curves $\Hom_n(E)$ form a family of equally oriented parabolas\footnote{This term is an abuse of language; in general the curves do not need to be parabolas since $\delta_P$ may be different from $2$.} with fold points on $\beta=-b_{22}\alpha^{\delta_P}$. It follows from (\ref{eq2}) that for positive $\alpha \to 0$ we have $n \to \infty$ and thus infinitely curves corresponding to $E$-homoclinics occur.
\medbreak

\item We give a proof for Case 5, where $k, a_{21},b_{22}>0$. We suggest the reader follows this proof observing Figure \ref{double_loop}.  For $n\geq n_1$, let $(\alpha_\star,\beta_\star) \in \Hom_{n}(E)$ the tip point of $ \Hom_{n}(E)$, so $\alpha_\star>0>\beta_\star$. For a fixed $\beta\leq \beta_\star<0$, let $(\alpha, \beta)  \neq (\alpha_\star,\beta_\star) $ such that $\alpha<\alpha_\star$. 
\medbreak
The set  $\Psi_{PE}^{-1}(W^s_{\loc}(E)\cap \Sigma_E^{\inn})$ is a parabola in $\Sigma_P^\out$, opened below and with its vertex in $\Sigma_P^{\out(+)}$. By Lemma \ref{helix} (hypothesis (2) holds), its pre-image in $\Sigma_P^\inn$ under $\Pi_P$ is a double helix accumulating on $W^s_{\loc}(P)\cap \Sigma_P^\inn$. 
Since $(\alpha, \beta)  \neq (\alpha_\star,\beta_\star) $ and $\alpha<\alpha_\star$, there is at least one piece of line within the set defined by $(\Pi_P^{-1} \circ \Psi_{PE}^{-1})(W^s(E))$ crossing transversely the cuspidal line corresponding to $(\Psi_{EP}\circ \Pi_E\circ \Psi_{PE}) (W^u(P))$ whose cusp point corresponds to $W^u(E)$.  This piece of line contained in $(\Pi_P^{-1} \circ \Psi_{PE}^{-1})(W^s(E))$ is mapped by  $\Psi_{EP}^{-1}$ into a segment crossing transversely $E^{cu}$ in $\Sigma_E^\out$. By continuity, its pre-image under $\Pi_E$ is arbitrarily close (for $\alpha \to \alpha_\star$) to $W^s_{\loc}(E)\cap \Sigma_E^\inn$ and thus to the original parabola $\Psi_{PE}^{-1}(W^s_{\loc}(E)\cap \Sigma_E^{\inn})$ in $\Sigma_P^\out$. This means that in $\Sigma_P^\inn$ there is at least one extra helix whose fold point is arbitrarily close to the original one, as illustrated in Figure \ref{double_loop}. 
\begin{figure}[ht]
\begin{center}
\includegraphics[width=15cm]{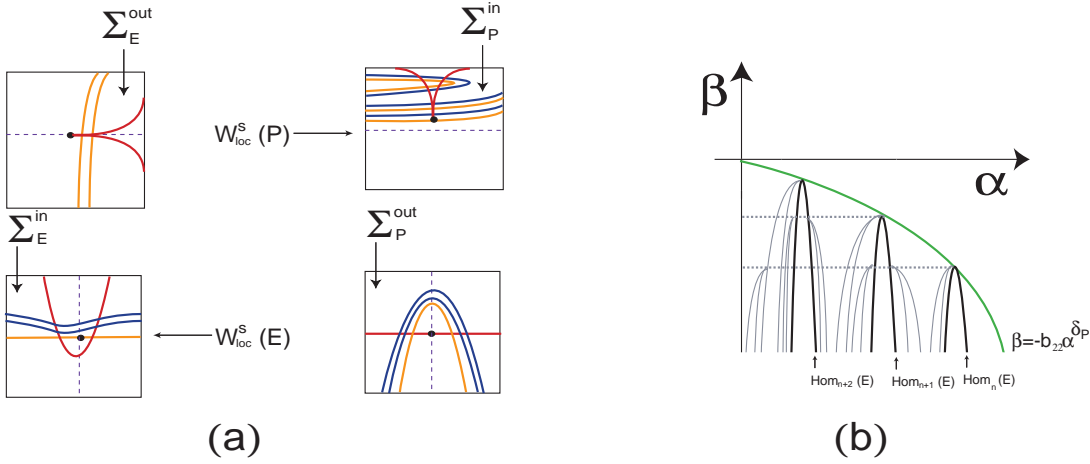}
\end{center}
\caption{\small Case 5: Illustration of the existence of infinitely many curves $\Hom_{n_2, n_1}(E)$ in the phase space (a) and the bifurcation diagram (b). }
\label{double_loop}
\end{figure}

 Therefore, for $\beta<\beta_\star$ fixed and decreasing $\alpha<\alpha_\star$, each point of $\Hom_{n}(E)$ is accompanied by an extra point. Extra helices appear when the cusp point in $\Sigma_P^\inn$ moves vertically in the lower direction, cutting $(\Pi_P^{-1} \circ \Psi_{PE}^{-1})(W^s(E))$. 
Since these curves vary continuously, for $\alpha \rightarrow \alpha_\star^-$  and $\beta \rightarrow \beta_\star^-$, we find a finite number of curves $\Hom_{n,m}(E)$ bifurcating from $\Hom_{m}(E)$, corresponding to homoclinics of $E$ that take two turns around the remnant of the heteroclinic cycle, with $m \geq n \geq n_1$. By construction the bifurcation occurs at the intersection of $\Hom_m(E)$ and $\beta=\beta_*$.  When $m=n$, we can see the extra homoclinic curves to $E$ only on the left side. An illustration of this situation may be found in 
Figures \ref{double_loop} and \ref{multipulses}. 

The proof for Case 2 is completely analogous (with $\beta>0$), except that the subsidiary curves bifurcate on the inside of the original curve $\Hom_n(E)$. 
\medbreak
\item For Cases 3 and 4, the shape of $W^s(E)$ in $\Sigma_P^\out$ is different from the previous item: it does not intersect $W^u_\loc(P)$, but is a parabola opened above and completely contained in the upper half of $\Sigma_P^\out$.

We give a proof for Case 3, where $k, a_{21}>0$ and $b_{22}<0$. Let $(\alpha_\star,\beta_\star) \in \Hom_{m}(E)$ be the tip point of $ \Hom_{m}(E)$ for $m \geq n_1$, so $\alpha_\star,\beta_\star>0$. For a fixed $\beta>0$, let $\alpha > \alpha_\star$. 
The set  $\Psi_{PE}^{-1}(W^s_{\loc}(E)\cap \Sigma_E^{\inn})$ is a parabola in $\Sigma_P^\out$, opened above and with its vertex in $\Sigma_P^{\out(+)}$. Its pre-image in $\Sigma_P^\inn$ under $\Pi_P$ is a curve with a fold point. Since $(\alpha, \beta)  \neq (\alpha_\star,\beta_\star) $ such that $\alpha < \alpha_\star$, there is at least one piece of line within the set $(\Pi_P^{-1} \circ \Psi_{PE}^{-1})(W^s(E))$ arbitrarily close to the point corresponding to $W^u(E)$. This piece of line contained in $(\Pi_P^{-1} \circ \Psi_{PE}^{-1})(W^s_\loc(E))$ is mapped by  $\Psi_{EP}^{-1}$ into a segment crossing transversely $E^{cu}$ in $\Sigma_E^\out$ and arbitrarily close to $W^u(E)\cap \Sigma_E^{\out}$. By continuity, its pre-image under $\Pi_E$ is close to $W^s_{\loc}(E)\cap \Sigma_E^\inn$ and thus to the original parabola $\Psi_{PE}^{-1}(W^s_{\loc}(E)\cap \Sigma_E^{\inn})$ in $\Sigma_P^\out$. This means that in $\Sigma_P^\inn$ there is at least one extra curve whose fold point is close to the original one.  Therefore, for $\beta>\beta_\star$ fixed, each point of $\Hom_{m}(E)$ is accompanied externally by an extra point.

\medbreak

\item In this proof we follow \cite[Theorem 3.2.12]{Wiggins}. Translated into our terminology it states that when an $E$-homoclinic cycle is broken, a periodic solution is created and it is 
\begin{itemize}
 \item a sink if $\lambda_1>\mu$ and $\lambda_2>\mu$;
 \item a saddle if $\lambda_1+\lambda_2>\mu$ and ($\lambda_1<\mu$ or $\lambda_2<\mu$);
 \item a source if $\lambda_1+\lambda_2<\mu$.
\end{itemize}
By \textbf{(H1)} we have $\lambda_1<\mu$, so the periodic solution cannot be a sink. Therefore, it is a saddle if $\lambda_1+\lambda_2>\mu$, and a source otherwise.
\medbreak
Now we have to determine on which side of the curves the periodic solutions appear. We argue for Case 5, where $a_{21}, b_{22}, k >0$, the other cases are similar. 
For $\alpha>0>\beta$, the set $\Pi_P^{-1}(W^s(E)\cap \Sigma_P^{\out})$ in $\Sigma_P^\inn$ is a double helix. For a fixed $\beta<0$, when $\alpha \rightarrow 0$, the point $W^u(E) \cap \Sigma_P^\inn$ intersects the helix infinitely many times, as illustrated in Figure \ref{periodic_solution}. Everything inside the double helix is mapped below $W^s(E)$ in $\Sigma_P^\out$ and thus leaves the neighbourhood $\mathcal{U}$, while points outside the double helix follow the heteroclinic cycle. 
So when $\alpha>0$ changes, the following happens to the point $W^u(E) \cap \Sigma_P^\inn$ -- see Figure \ref{periodic_solution}:
\medbreak
\begin{figure}[ht]
\begin{center}
\includegraphics[width=7cm]{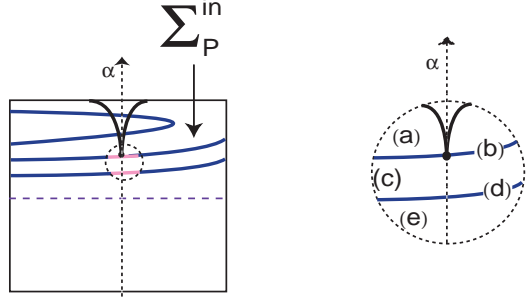}
\end{center}
\caption{\small Intersections of the helix $\Pi_P^{-1}(W^s(E)\cap \Sigma_P^{\out})$ with $W^u(E)$ in $\Sigma_P^\inn$, when $\alpha \to 0$. The $E$-homoclinics observed in (b) and (d) differ in the number of half twists of $W^s(E)$ along the cycle.}
\label{periodic_solution}
\end{figure}
\begin{itemize}
 \item[(a)] it falls outside the region bounded by the double helix $W^s(E)$, \emph{i.e.}\ it follows the cycle;
 \item[(b)] it falls onto the double helix $W^s(E)$, and an $E$-homoclinic cycle is created;
 \item[(c)] it is inside the region bounded by the double helix $W^s(E)$, \emph{i.e.}\ it leaves $\mathcal{U}$;
 \item[(d)] it falls onto the double helix $W^s(E)$, and another $E$-homoclinic cycle is created;
 \item[(e)] it falls outside the region bounded by the double helix $W^s(E)$, \emph{i.e.}\ it follows the cycle.
\end{itemize}
\medbreak
We argue that a periodic solution appears in regions (c) and (e), \emph{i.e.}\ for decreasing values of $\alpha$. This follows from the considerations in \cite{Shilnikov_book}, as we proceed to explain:
\begin{itemize}
\item the saddle value of $E$ is positive;
\item the $E$-homoclinic is not contained in the strong stable manifold of $E$;
\item with exception to the tip point, the center-unstable manifold of $E$ is transverse to $W^s(E)$.
\end{itemize}
\medbreak
Thus, when $\alpha \rightarrow 0$, a single periodic solution is born for consecutive values of $\alpha$ for which we observe an $E$-homoclinic. With the exception of the fold point, at which we have an inclination flip bifurcation, the subsequent boundaries that $W^u(E)$ passes through, belong to the same $n \in \NN$, in the sense that the corresponding homoclinic cycle takes $n$ turns around $P$. They differ by a half twist of $W^s(E)$ along the $E$-homoclinic, the first one falls into the cylinder case, the second one into the M\"obius band case. That is, in the former case the two-dimensional unstable manifold of the emergent periodic solution is diffeomorphic to a cylinder, while in the latter it is a non-orientable surface. Figure 3.2.19 in \cite{Wiggins} illustrates how this proves our claim.
\end{enumerate}
\end{proof}

We remark that iterating the arguments in (3) and (4) leads to the following result:
%
\begin{corollary}
Under the hypotheses of Theorem \ref{thm-E-hom}, for Cases 2--5, there are infinitely many curves $\Hom_{m_1, \ldots ,m_k}(E)$ in the parameter space for any $k \in \NN$ and $m_k>m_{k-1}>\ldots >m_1 \geq n_1 \in \NN$. In Cases 3 and 4, their tip points coincide with the tips of $\Hom_{m_1}(E)$. In Cases 2 and 5, each curve $\Hom_{m_1, \ldots ,m_k}(E)$ bifurcates into two curves at each intersection point between $\Hom_{m_1, \ldots ,m_k}(E)$ and $\beta=\beta_\star$ where $\beta_\star$ is the fold point of $\Hom_{m}(E)$ for all $m< m_1$. 
\end{corollary}

The differences in the argumentation for Cases 2 and 5 compared to 3 and 4 are reflected in the fact that for the latter we leave the region of $E$-homoclinics by fixing $\beta>0$ (Case 3) or $\beta<0$ (Case 4) and then considering $\alpha \to 0$. Figure \ref{multipulses} illustrates this.
In particular there are regions in the parameter space with positive Lebesgue measure where we observe chaos and non-hyperbolic behaviour. This is in contrast to the findings for Case 1 described in \cite{MP}.

\begin{figure}
\begin{center}
\includegraphics[width=\textwidth]{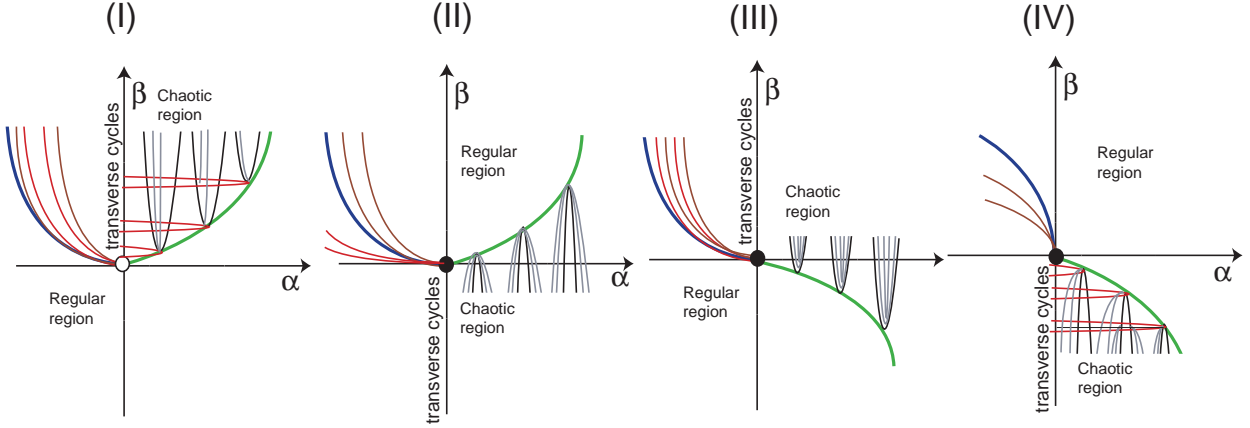}
\end{center}
\caption{\small Plausible bifurcation diagrams for the Cases given in Theorem \ref{main-thm}. The bold point at the origin means that there is chaos in the organizing center. Green line: $\beta=-b_{22}\alpha^{\delta_P}$. Black line: $\Hom_n(E)$. Gray lines: subsidiary $E$-homoclinics. Red line: $\Het_n(PE)$. Blue line: $\Hom(P)$ ($\beta=\tilde{k} \alpha^{\frac{1}{\delta_1}}$). Brown lines: $\Hom_n(P)$.}
\label{multipulses}
\end{figure}

\subsection{Homoclinic cycles to $P$}
In the next result, we show that there are infinitely many curves in the parameter space with homoclinic tangencies to $P$. These play a crucial role in the transition to chaos which characterizes the boundary crisis. There is a first homoclinic tangency, which gives rise to an infinite sequence of homoclinic tangencies accumulating on each other. After the first tangency the sets $W^s(P)$ and $W^u(P)$ accumulate on each other. Once we cross the curve $\Hom(P)$ the invariant manifolds of $P$ form a homoclinic tangle, meaning that the first homoclinic tangency can be seen as the onset of regular chaos. Again the proof of item (1) is based on \cite[Section 3.3]{Champneys2009}.

\begin{theorem}\label{thm-P-hom}
Consider a vector field $f_0$ satisfying $\textbf{(H1)}-\textbf{(H5)}$ and a generic two-parameter family of $C^\infty$ vector fields $f (x,\alpha,\beta)$ unfolding $f_0$, for $\alpha, \beta\neq 0$ sufficiently small. Then, the dynamics satisfy:
\begin{enumerate}
\item For Cases 2--5, there is a curve $\Hom(P)$, defined by $\alpha =\tilde{k}\beta^{\delta_1}$, for $\beta>0$ and $\tilde{k}<0$.
\item For Cases 2, 4 and 5, there are infinitely many curves $\Hom_n(P)$ accumulating (from the chaotic region) on the curve $\Hom(P)$. For Case 3, there are infinitely many curves $\Hom_n(P)$ accumulating (from the regular region) on the curve $\Hom(P)$.
\end{enumerate}
\end{theorem}

\begin{proof}
\begin{enumerate}
\item The set $W^u(P)\cap \Sigma_P^\out$ is given by $\rho=0$ and $\psi$ is free. This line is mapped by $\Psi_{PE}$ into the set parametrized by
$
(y,z)=(b_{11}\psi+\beta\xi_2, k\psi^2+\beta+\nu\beta\psi),
$
which is mapped by $\Pi_E$ into
$$
(x,y)= ((k\psi^2+\beta+\nu\beta\psi)^{\delta_1}, (b_{11}\psi+\beta\xi_2)(k\psi^2+\beta+\nu\beta\psi)^{\delta_2}) , \qquad \text{where} \qquad k\psi^2+\beta+\nu\beta\psi\geq 0.
$$
The $z$-component of its image under $\Psi_{EP}\circ \Pi_E\circ \Psi_{PE}$ is given by
$$
a_{21}(k\psi^2+\beta+\nu\beta\psi)^{\delta_1} +a_{22}(b_{11}\psi+\beta\xi_2)(k\psi^2+\beta+\nu\beta\psi)^{\delta_2}+\alpha.
$$
In $\Sigma_P^\inn$, the stable manifold $W^s_\loc(P)$ is given by $z=0$, so up to higher order terms, we are looking for solutions of 
\begin{equation}
\label{alpha1}
 \alpha(\beta, \psi)=-a_{21}(k\psi^2+\beta+\nu\beta\psi)^{\delta_1},
 \end{equation}
  which has solutions if and only if  
$\sign(\alpha)=-\sign \left(a_{21}(k\psi^2+\beta+\nu\beta\psi)^{\delta_1} \right)$.  
Moreover, $W^s(P)$ and $W^u(P)$ have a tangency if and only if $\frac{\partial \alpha}{\partial \psi}(\beta,\psi) =0$ which implies that
$$
-a_{21}\delta_1(2k\psi+\nu\beta)(k\psi^2+\beta+\nu\beta\psi)^{\delta_1}=0.
$$
This implies that $\beta=\frac{-2k\psi}{\nu}$ (the other terms cannot be zero).  Replacing it in (\ref{alpha1}), it follows that $\alpha=-a_{21}k^\star \beta^{\delta_1}$, where $k^\star \in \RR^+ $.
\medbreak

\item We start with Case 3. Let $\alpha<0<\beta$. We first recover the curve $\Hom(P)$ geometrically. The set $W_\loc^u(P)$ in $\Sigma_E^\inn$ is of the form $\mathcal{P}_\beta^+$, so it does not intersect $W_\loc^s(E)$. Therefore, by Lemma \ref{helix}, its image under $\Psi_{EP} \circ \Pi_E$ in $\Sigma_P^\inn$ is a curve with a fold point, detached from $W^u_\loc(E)$. For some pair $(\alpha_\star,\beta_\star)\in \Hom(P)$, the fold curve is tangent to $W_\loc^s(P)$. 
The key argument is the following: the pre-image under $\Psi_{EP}\circ \Pi_E\circ\Psi_{PE}$ of $W_\loc^s(P)\subset \Sigma_P^\inn$ is a parabola opened above and tangent to $W_\loc^u(P)$. Its pre-image in $\Sigma_P^\inn$ is a helix accumulating on $W_\loc^s(P)$. That means that for appropriate $\alpha>\alpha_\star$ or $\beta>\beta_\star$ we find tangencies $\Hom_n(P)$ accumulating on $\Hom(P)$. They lie outside the chaotic region depicted in Figure \ref{multipulses}. A given tangency becomes transverse when $\alpha$ decreases, for larger $\alpha$ it disappears.

\begin{figure}
\begin{center}
\includegraphics[height=8cm]{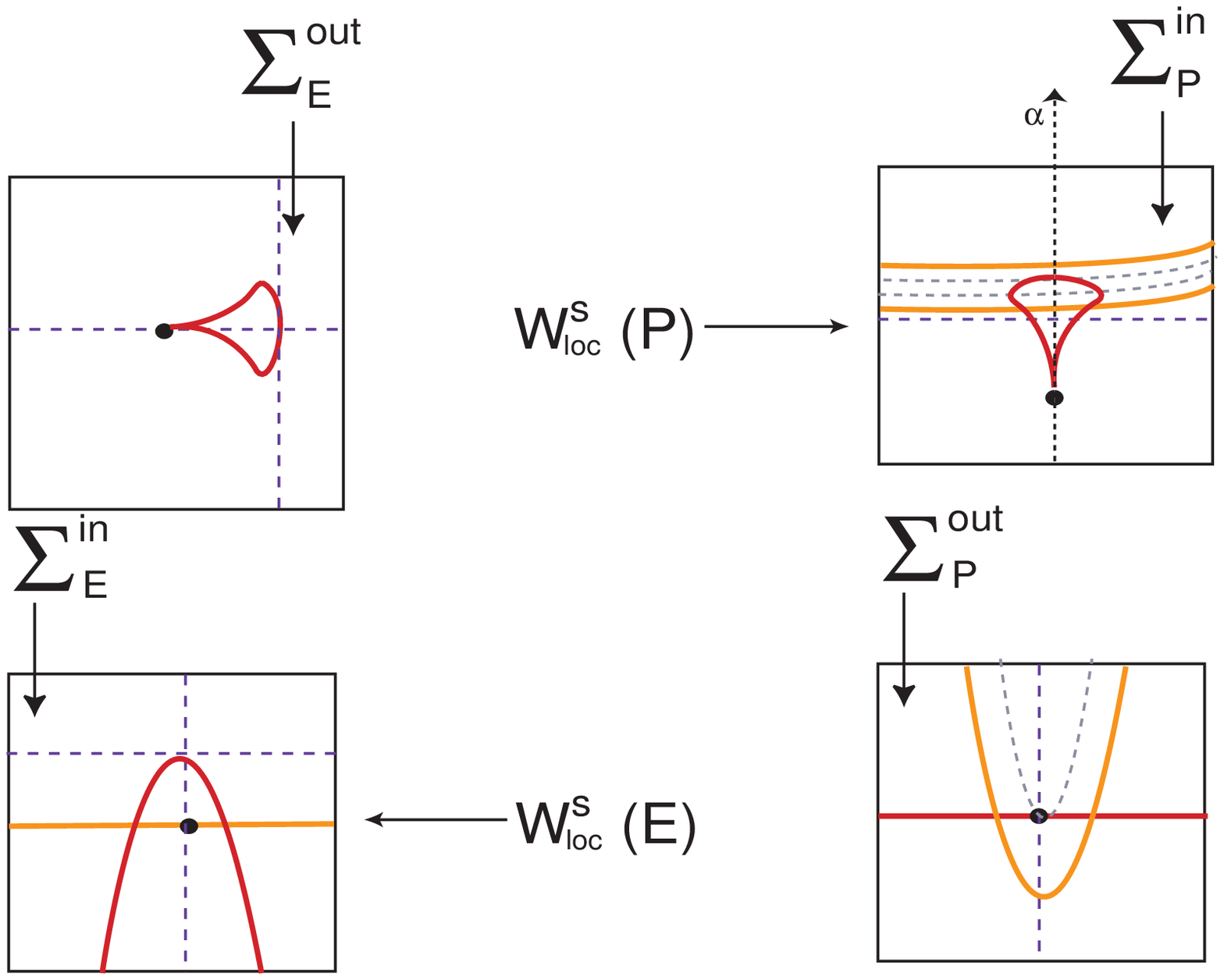}
\end{center}
\caption{\small Illustration of Theorem \ref{thm-P-hom} (2), Case 4: creation of $P$-homoclinic multipulses.}
\label{multipulses-case4}
\end{figure}

The proof for Case 4 is similar and is depicted in Figure \ref{multipulses-case4}. The tangency $\Hom(P)$ occurs for $\alpha_\star<0<\beta_\star$, when the image of $W^u(P)$ in $\Sigma_P^\inn$ is a connected cusp curve tangent to $W_\loc^s(P)$. By the same reasoning as before, the pre-image of $W^s_\loc(P)$ under a full return is a helix accumulating on $W_\loc^s(P)$, so there are countably many tangencies $\Hom_n(P)$ accumulating on $\Hom(P)$. Again these occur for $\alpha>\alpha_\star$, but, in contrast to Case 3, this corresponds to the chaotic region. In Case 5, the arguments are again the same, but it is worth noting that the helix that is the pre-image of $W^s(P)$ in $\Sigma_P^\inn$ under a full return only appears for $\alpha<\alpha_\star$. Otherwise, $W^s(P)$ falls into the lower region of $\Sigma_P^\out$ and thus has no further pre-image. This is similar also in Case 2.
\end{enumerate}
\end{proof}

As in the previous theorem, straightforward iteration of the arguments in (2) yields curves $P$-homoclinic solutions with more than two pulses. The result is a consequence of the Newhouse phenomena and/or H\'enon attractors near the tangency, characterized by the existence of wild sets in the corresponding flow.

\begin{corollary}
Under the hypotheses of Theorem \ref{thm-P-hom}, there are infinitely many curves $\Hom_{n_1,\ldots,n_{k-1}}(P)$ in the parameter space for $k, n_1,\ldots,n_{k-1} \in \NN$. They accumulate on $\Hom(P)$ in the same way as above.
\end{corollary}

After the bifurcating curve $\Hom_k(P)$ the sets $\overline{W^u(P)}$ and $\overline{W^s(P)}$ are contained in a chaotic attractor. For small $\alpha$ after the first homoclinic bifurcation to $P$ there is a region of preturbulence and Newhouse phenomena, that may be attracting or repelling according to the eigenvalues of \textbf{(H2)}. One finds a region of multistability where the attractor coexists with attracting/repelling periodic solutions. They constitute an essential ingredient to the characterization of the boundary crisis, whose  full description is unreachable as we discuss in Section \ref{chaos-dynamics}.

\subsection{Heteroclinic tangencies from $P$ to $E$}
Finally, we give a corresponding result about tangent connections $[P \to E]_t$ taking more and more turns around the heteroclinic cycle. For different cases these occur in different quadrants of the parameter space.
\begin{theorem}\label{thm-PE-het}
Consider a vector field $f_0$ satisfying $\textbf{(H1)}-\textbf{(H5)}$ and a generic two-parameter family of $C^\infty$ vector fields $f (x,\alpha,\beta)$ unfolding $f_0$, for $\alpha, \beta\neq 0$ sufficiently small. Then the following holds:
\begin{enumerate}
\item In Cases 2 and 5 there are infinitely many curves $\Het_{n}(PE)$ whose tip points coincide with those of  $\Hom_{n}(E)$.
\item In Cases 2 and 4 there are infinitely many curves $\Het_{n}(PE)$ intercalating with $\Hom_n(P)$.
\item In Case 3 there are infinitely many curves $\Het_{n}(PE)$ accumulating on the $\alpha$-axis. 
\end{enumerate}
\end{theorem}
\begin{proof}
\begin{enumerate}
 \item We treat Case 5, illustrated in Figure \ref{heteroclinic_tangency}, Case 2 is analogous with $\beta>0$. First recall the geometric arguments for the existence of the curves $\Hom_n(E)$ from Theorem \ref{thm-E-hom}. Then we show that their tip points are also the tip points of curves corresponding to heteroclinic tangencies from $P$ to $E$, turning $n$ times around $P$ in between. So, as in Theorem \ref{thm-E-hom}, let $\alpha>0>\beta$. Then the pre-image of $W_\loc^s(E)$ in $\Sigma_P^\inn$ is a double helix accumulating on $W^s_\loc(P)$. At the tip point $(\alpha_\star, \beta_\star)$ of $\Hom_n(E)$ its fold point coincides with the cusp point $W^u(E)$.
 
 On the other hand, the image of $W_\loc^u(P)$ in $\Sigma_P^\inn$ is a cusp curve attached to $W_\loc^u(E) \cap \Sigma_P^\inn$. So for $\alpha<\alpha_\star$ there are two values $\beta_1>\beta_\star>\beta_2$ corresponding to heteroclinic tangencies. These points define the curve $\Het_n(PE)$.
 \medbreak
 \item We treat Case 4, Case 2 is analogous. For $\beta>0>\alpha$ the set $W^u(P)$ in $\Sigma_E^\inn$ is a parabola opened below that intersects $W_\loc^s(E)$ twice. Therefore, by Proposition \ref{dyn_E_proposition} (2), its image is a connected cusp curve attached to $W^u_\loc(E)$ in $\Sigma_E^\out$, and thus also in $\Sigma_P^\inn$. Consider $(\alpha_\star,\beta_\star) \in \Hom(P)$. Then increasing $\alpha>\alpha_\star$ there are countably many tangencies between the helix that is the pre-image of $W^s(E)$ and the top of the cusp curve that is $W^u(P)$ in $\Sigma_P^\inn$, creating the curves $\Het_n(PE)$. Trajectories take $n \in \NN$ turns around the remnant of the heteroclinic cycle as $\alpha \to \alpha_\star$ from above.
 \begin{figure}[ht]
\begin{center}
\includegraphics[width=16cm]{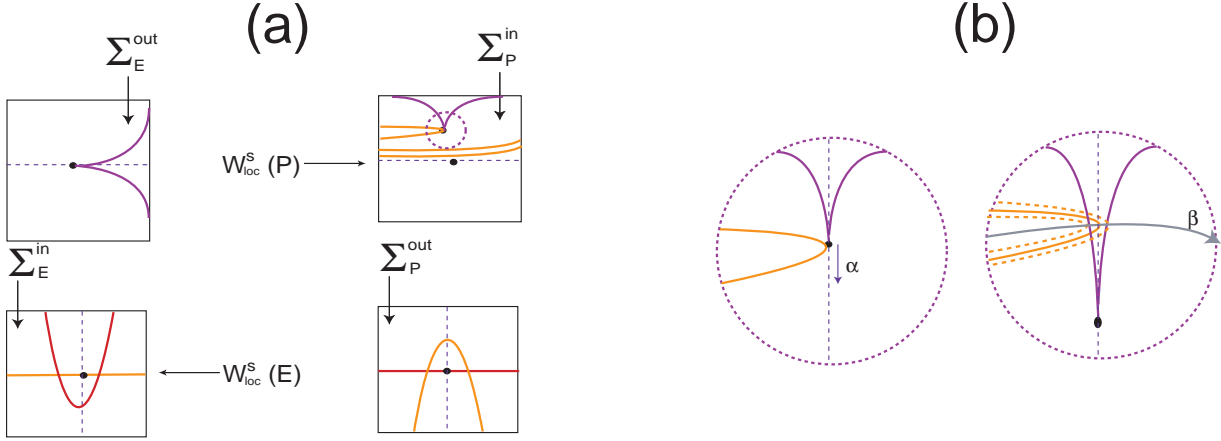}
\end{center}
\caption{\small Illustration of Theorem \ref{thm-PE-het} (1), Case 5: creation of $[P \to E]_t$ from $E$-homoclinic cycles at $(\alpha_\star, \beta_\star)$. (a) Cross sections. (b) Zoom of the global bifurcation: for $\alpha<\alpha_\star$ there are two values $\beta_1>\beta_\star>\beta_2$ corresponding to heteroclinic tangencies. The effect of changing $\beta$ is indicated by the arrow on the horizontal line.}
\label{heteroclinic_tangency}
\end{figure}
 Since the pre-image of $W_\loc^s(P)$ in $\Sigma_P^\out$ is a parabola that lies above the parabola given by $W^s(E)$ and touches $W_\loc^u(P)$, as in Figure \ref{heteroclinic_tangency}, this leads to intercalated helices in $\Sigma_P^\inn$.
 \medbreak
 \item For $\beta=0$ and fixed $\alpha<0$, the pre-image of $W_\loc^u(P)$ under $\Psi_{EP} \circ \Pi_E \circ \Psi_{PE}$ is a cuspidal curve in $\Sigma_P^\inn$ with its cusp point below $W_\loc^s(P)$. The pre-image of $W^s_\loc(E)$ under $\Pi_P$ is a helix accumulating on $W_\loc^s(P)$. For $\beta>0$, this helix becomes a curve with a fold point, touching the cusp infinitely many times when $\beta \to 0$. In particular, this yields a sequence $\{\beta_n\}_{n \in \NN}$, such that we have heteroclinic tangencies for $(\alpha, \beta_n)$. Since this holds for every $\alpha<0$, by continuity we obtain the curves $\Het_n(PE)$.
\end{enumerate} 
\end{proof}

\section{The transition to chaotic dynamics}
\label{chaos-dynamics}

{In this section we give an overview of the information we have acquired in the previous sections on the dynamics for the different parameter configurations of an EP1t-cycle. We also discuss the results and possible further developments.

\subsection{On the organizing center} 
As stated in Theorem \ref{main-thm}, in three of the eight situations arising from the choice of sign for the parameters $a_{21}, k$ and $b_{22}$, there is a suspended horseshoe accumulating on $\Gamma$, in the organizing center $\alpha=\beta=0$.  We comment on the other two cases first.

In Case 1, where we have $a_{21}<0$ and both $k,b_{22} \gtrless 0$, the singular cycle is isolated due to the loss of trajectories along the connection $[E \to P]$, so there is no chaos in the organizing center. This situation was noted and studied in more detail in \cite{MP}, where it also becomes clear that for $a_{21}<0$ there are open regions of $(\alpha, \beta)$ near $(0,0)$ for which $E$- and $P$-homoclinic cycles coexist; the same does not happen for Cases 2--5, where $a_{21}>0$. In Case 2 the cycle is also isolated, this time all trajectories leave its neighbourhood in the transition $[P \to E]_t$. This is because trajectories passing around $P$ subsequently hit $\Sigma_E^{\inn (-)}$, the lower part of the incoming cross section near $E$.
The other Cases 3--5 display shift dynamics in the organizing center -- with different properties with respect to the conjugacy to a full shift, depending on the transition from $P$ to $E$: when $W^u(P)$ is folded inward ($k>0$, Cases 3 and 5), chaos may exist regardless of $b_{22}\gtrless 0$. For an outward fold ($k<0$, Case 4), however, chaos is incompatible with an inclination flip ($b_{22}<0$). The differences in these types of shift dynamics of the organizing center stem from the shape of the inner part of $W^u(P)$, which is connected or not. Our Case 4 has dynamical properties as those described by Tresser \cite[Section V]{Tresser}.

We also proved the existence of a sequence of hyperbolic invariant sets accumulating on the cycle for which the dynamics of the first return map are topologically conjugate to a full shift on finitely many symbols. The closure of their union is not uniformly hyperbolic. Any subset containing a finite number of these horseshoes is uniformly hyperbolic, persisting for small smooth perturbations.  When $(\alpha, \beta)\neq (0,0)$, many of the heteroclinic circuits will be removed in saddle-node type bifurcations (combined with period doubling). 
\medbreak
The dynamics of the (regular) horseshoe is chaotic and hyperbolic on the product of two Cantor sets and its complement is dense in the neighbourhood of $\Gamma$, so that most nearby initial conditions escape after a finite number of iterations. A natural question is how much richer is the dynamics of $\Lambda$ in the case that it is a cusp horseshoe?  Theorems \ref{main-thm} and \ref{thm-hyp} combined raise another question: how does the arrangement of the invariant manifolds of $\Lambda$ vary with the $W^{cu}(E)$ and $W^s(P)$?  The results suggest that the topological organization of the phase space near $\mathcal{U}$ is governed by $W^{cu}(E)$ and $W^s(P)$. We also expect these manifolds to organize the structure of the basin boundaries of the bifurcating sinks/sources of Theorems \ref{thm-E-hom}, \ref{thm-P-hom} and \ref{thm-PE-het}.

\medbreak
\subsection{Perturbing the vector field} 
Theorems \ref{thm-E-hom}, \ref{thm-P-hom} and \ref{thm-PE-het} characterize the boundary crisis associated to the transitive set $\Lambda$ (introduced in Section \ref{organizing_center}) and extend the results in \cite{Champneys2009, MP}. In this regard, our main contribution is the description of multipulse homoclinics to $E$ and $P$, as well as heteroclinic tangencies from $P$ to $E$.  The $E$-homoclinic cycles are distinguished by the number of revolutions around $P$. Curves marking the creation/disappearance of $E$- and $P$-homoclinics occur for opposite signs of $\alpha$. Numerically, the multipulses characterize the non-hyperbolicity of the region, and they have the same effect as Cocoon bifurcations: accumulation of parameter values for which heteroclinic tangencies between nodes coexist with solutions of long periods.

We elaborate primarily on the unfolding dynamics for Case 2, the other cases can be inferred from this. Fix $\beta>0$ and decrease $\alpha$. Once we cross the curve $\beta=-b_{22}\alpha^{\delta_P}$, we find homoclinic cycles to $E$ that break up into periodic solutions which should persist until $\alpha=0$. On the $\beta$-axis we find a singular cycle with the same properties as in \cite{PR}. For small $\alpha<0$, the dynamical properties of this transitive set should persist until multipulses to $P$ appear. This region is characterized by instability until the last tangency at $\Hom(P)$. After this there are no recurrent dynamics. For $\beta<0$ the dynamics are trivial as suggested in Figure \ref{multipulses}.

Figure \ref{multipulses-case4} yields additional information on the transition to chaos for Case 4: the $P$-homoclinic multipulses from Theorem \ref{thm-P-hom} become transverse for fixed $\beta>0$ and increasing $\alpha$. They then persist, until for some $\alpha>0$ they are annihilated in a heteroclinic connection $[E \to P]$. Analogously, the heteroclinic tangencies $[P \to E]_t$ from Theorem \ref{thm-PE-het} become transverse for increasing $\alpha$ and culminate in the $E$-homoclinics given through $\Hom_{n_1,\ldots,n_k}(E)$.

In Cases 3 and 4, when the cusp point coincides with a homo- or heteroclinic connection, there exists a subset in $ \Sigma_P^{\inn}$ where the map $R$ is topologically semi-conjugate to a Bernoulli shift, a kind of \emph{singular horseshoe} considered in \cite{LP}. In these cases, the semi-conjugacy fails to be a conjugacy  because of the pinching of the stable directions caused by the lack of injectivity of $R$ on a given fiber.

\subsection{Other dynamics}
We close with a remark on \emph{switching dynamics} \cite{ALR, CastroLohse2015}: in each of the Cases 2--5 there are parameter regions where there exists a heteroclinic $EP1$-cycle, \emph{i.e.}\ it is of codimension one instead of two, because the intersection $W^u(P) \cap W^s(E)$ is transverse. This occurs in
\begin{itemize}
 \item Cases 2 and 4 when $\beta>0$ and $\alpha=0$,
 \item Cases 3 and 5 when $\beta<0$ and $\alpha=0$.
\end{itemize}
Then, there are two connections from $P$ to $E$, say ${\gamma}_0$ and $\tilde{\gamma}_0$, and in fact we have a heteroclinic network (with three connections).
For $\beta=0$ it collapses into the singular cycle $\Gamma$. By the same arguments as in \cite{ALR} (one saddle is a rotating node) it is possible to conclude that there is infinite switching near this network, meaning that all possible sequences of following $\gamma_0$ and $\tilde{\gamma}_0$ are shadowed by real trajectories in any neighbourhood of the network. Roughly speaking, this happens because of the following: small rectangles in $\Sigma_E^\inn$, centered around the connections, are mapped to vertical strips touching $W^s_\loc(P)$ in $\Sigma_P^\inn$, and thus, to the region inside a double helix in $\Sigma_P^\out$, accumulating on $W^u_\loc(P)$. Each such region contains smaller rectangles contained in both of the original ones and the process may be iterated.
}

\section{Concluding Remarks}\label{conclusion}

We have investigated the dynamics near a singular $EP1t$-cycle $\Gamma$, \emph{i.e.}\ a heteroclinic cycle with a quadratic tangency between $W^u(P)$ and $W^s(E)$, as the organizing center of a two-parameter bifurcation scenario.  Throughout our analysis, we have  focused on the case where $E$ has real non-resonant eigenvalues and $P$ has positive Floquet multipliers. 
Without any additional property (reversibility, free-divergence) of the initial vector field, we studied global bifurcations of homoclinics and associated periodic solutions occuring in the process of the dissolution of shift dynamics. Moreover, via a qualitative analysis, we partially described the boundary crisis associated to a codimension two singular cycle.

We characterized the homoclinic bifurcations of periodic solutions that form part of a larger topologically transitive invariant set. We compared the bifurcation scenario with others in which regular and singular horseshoes break up. These are triggered by homoclinic tangencies or saddle-node bifurcations of periodic solutions. The transition to chaos is complicated but, with this paper, assuming some non-restrictive hypotheses, we were able to give a qualitative description of what is going on and we partially described the mechanism leading to destruction of a snaking region -- see the open problems \textbf{(Q1)--(Q3)} of Knobloch \cite{Knobloch_open}. When we pass the first bifurcation line, sudden changes in the non-wandering set in $\mathcal{U}$, a small neighbourhood of the cycle, occur.

This kind of singular cycle, although exhibiting complicated behaviour, is sufficiently simple to be treated analytically. In contrast to \cite{ALR, LR2015}, in which the authors find \emph{instant chaos}, the boundary crisis near $\Gamma$ is characterized by an accumulation of homoclinic cycles and heteroclinic tangencies which give rise to periodic solutions making more and more turns around the original cycle. These solutions should accumulate on the singular horseshoe described by \cite{MP, PR} and others. Besides these global bifurcations, there might exist other global and local bifurcations which cannot be detected by our analysis. Assuming some conditions on the dissipativeness of the first return map, it would be interesting to specify more properties about the existence of several attractors in the unfolding of the cycle and the mechanisms through which their basins of attraction interact and change their size/shape.  The construction of explicit polynomial vector fields whose flow satisfies \textbf{(H1)--(H5)} is an open problem.  
We defer these tasks to future work.

\section*{Acknowledgements} 

We thank the anonymous referees for the careful reading of the manuscript and the useful suggestions that helped to improve the work.
 CMUP is supported by the European Regional Development Fund through the programme COMPETE
and by the Portuguese Government through the Funda\c{c}\~ao para a Ci\^encia e a Tecnologia (FCT) under the project PEst-C/MAT/UI0144/2011. Support through the FCT grants SFRH/BPD/84709/2012 (for AR) and Incentivo/MAT/UI0144/2014 (for AL) is gratefully acknowledged. Part of this work has been written during AR stay in Nizhny Novgorod University supported by the grant RNF 14-41-00044.


\end{document}